\documentclass[12pt]{amsart}
\usepackage{amscd}
\usepackage[all,cmtip]{xy}
\usepackage{tikz}
\usepackage{hyperref}
\usepackage{hyperendnotes}
\usepackage{latexsym,calligra,amsmath,amssymb,stmaryrd,MnSymbol}
\usepackage[parfill]{parskip} 
\newif\ifdviwin
\usepackage[english]{babel}
\usepackage{indentfirst}
\usepackage[mathscr]{eucal}
\usepackage{amssymb,amsmath,amsfonts}
\usepackage{graphicx}
\usepackage{amsmath,amscd}
\usepackage{xxcolor}
\newtheorem{theorem}{Theorem}[section]
\newtheorem{lemma}[theorem]{Lemma}
\newtheorem{corollary}[theorem]{Corollary}
\newtheorem{proposition}[theorem]{Proposition}

\newtheorem{definition}[theorem]{Definition}

\numberwithin{equation}{section}

\newcommand{\ms}[1]{\mathscr{#1}}
\newcommand{\mbf}[1]{\mathbf{#1}}
\newcommand{\mbb}[1]{\mathbb{#1}}
\newcommand{\mfr}[1]{\mathfrak{#1}}

\newcommand{\tn}[1]{\textnormal{#1}}
\newcommand{\fsz}[1]{\footnotesize{#1}}
\def\ee{\varepsilon}
\def\PP{\mathbb{P}}
\def\AA{\mathbb{A}}
\def\QQ{\mathbb{Q}}
\def\ZZ{\mathbb{Z}}

\def\mm{\mathfrak{m}}

\def\mm{\mathfrak{m}}

\numberwithin{equation}{subsection} 

\begin{document}

\title[Infinitesimal Structure of Chow Groups]{Infinitesimal Theory of Chow Groups via $K$-Theory, Cyclic Homology, and Relative Chern Character}

\author[Dribus, Hoffman, Yang]{Benjamin F. Dribus, J. W. Hoffman, Sen Yang}
\address{Dribus:Mathematics Department\\
William Carey University\\
Hattiesburg, Mississippi}
\address{Hoffman:Mathematics Department\\
Louisiana State University\\
Baton Rouge, Louisiana}
\address{Yang:Mathematical Sciences Center\\
Tsinghua University\\
Beijing China}
\email{BDribus@wmcarey.edu,hoffman@math.lsu.edu,syang@math.tsinghua.edu.cn}
\begin{abstract}
We examine the tangent groups at the identity, and more generally the formal completions at the identity, of the Chow groups of algebraic cycles on a nonsingular quasiprojective algebraic variety over a field of characteristic zero. We settle a question recently raised by Mark Green and Phillip Griffiths concerning the existence of Bloch-Gersten-Quillen-type resolutions of algebraic $K$-theory sheaves on infinitesimal thickenings of nonsingular varieties, and the relationships between these sequences and their ``tangent sequences," expressed in terms of absolute K\"{a}hler differentials.  More generally, we place Green and Griffiths' concrete geometric approach  to the infinitesimal theory of Chow groups in a natural and formally rigorous structural context, expressed in terms of nonconnective $K$-theory, negative cyclic homology, and the relative algebraic Chern character. \\

\smallskip
\noindent \textbf{Key Words:} Chow groups, algebraic $K$-theory, algebraic cycles, cyclic homology, Chern character, local cohomology, $K$-theory operations.\\

\smallskip
\noindent \textbf{Mathematics subject classification 2010:} 14C15, 19E15, 14C25, 19D55, 19L10, 14B15, 55S25.

\end{abstract}

\date{\today}

\maketitle


\section{Introduction.}\label{Intro} 

\subsection{Chow Groups; $K$-Theory; Tangent Groups.}
\label{subsecchowK}
The Chow groups $CH^p(X)$ of codimension-$p$ algebraic cycles modulo rational equivalence on a nonsingular quasiprojective algebraic variety 
$X$ over a field of characteristic zero remain poorly understood despite intensive study by algebraic geometers over the last half-century.  Significant early progress was made in the early 1970's by Bloch, Gersten, and Quillen, who established the existence of the fundamental isomorphism
\begin{equation}\label{equblochintro}
CH^p(X)\cong H^p\big(X,\ms{K}_p(X)\big),
\end{equation}
where $\ms{K}_p(X)$ denotes the sheaf of algebraic $K$-groups of the structure sheaf $\ms{O}_X$ of $X$, and where $H^p$ denotes Zariski sheaf cohomology.  The existence of such a relationship between algebraic cycles and algebraic $K$-theory had been conjectured by Bloch and Gersten before Quillen's higher $K$-theory was available to properly express it; the case $p=1$ is ``classical," while the case $p=2$ was proven by Bloch around 1972.  The general case was established soon afterward by Quillen in his foundational paper \cite{QuillenHigherKTheoryI72}.  Despite these developments, the structural details of the Chow groups remain largely inaccessible because $K$-theory is so challenging to compute.  

In response to this challenge, a number of authors have defined and studied simpler, ``linearized" versions of the Chow groups; namely, their {\it tangent groups at the identity}, or more generally, their {\it formal completions at the identity}.   This approach replaces the intractable global theory of Chow groups with a simpler infinitesimal theory, in a manner analogous to the use of Lie algebras to study Lie groups.  Particularly relevant to our present work are the early contributions of Van der Kallen  \cite{VanderKallenEarlyTK271}, Bloch \cite{BlochTangentSpace72}, \cite{BlochK2Cycles74}, \cite{BlochKTheoryGeometry72}, and Stienstra \cite{Stienstra83}, \cite{Stienstra85}, in this direction.   The case of the second Chow group $CH^2(X)$ invited special attention during this period as the ``first nonclassical case."  The simplest result arising from this study, a result of great historical importance, is Bloch's formula for the tangent group at the identity $TCH^2(X)$:
\begin{equation}\label{equblochinf}
TCH^2(X)\cong H^2(X, \Omega^1 _{X/\ZZ}),
\end{equation}
where $\Omega^1 _{X/\ZZ}$ is the sheaf of {\it absolute K\"{a}hler differentials} on $X$.   The isomorphism in equation \hyperref[equblochinf]{\ref{equblochinf}} is highly nontrivial, involving a ``logarithmic map" admitting vast generalization.  
 

\subsection{Approach of Green and Griffiths; a Basic Question.}
\label{subsecBGQquestionGG}

In a recent work \cite{GreenGriffithsTangentSpaces05}, Mark Green and Phillip Griffiths have introduced an interesting avenue of investigation into the infinitesimal structure of the Chow group, particularly their tangent groups at the identity.  They adopt a concrete geometric viewpoint, in which elements of $TCH^p(X)$ are viewed as ``first-order deformations of arcs of cycle classes."  The principal technical tools for their approach are the ``na\"{i}ve" Milnor version of algebraic Bloch's formula \hyperref[equblochinf]{\ref{equblochinf}} for $TCH^2(X)$, in the case where $X$ is a nonsingular quasiprojective complex algebraic surface.  This viewpoint, with its analogies to differential geometry, has obvious conceptual advantages, but also presents formidable technical difficulties. 

Prominent in the work of Green and Griffiths are special Cousin resolutions of sheaves on the variety $X$, analogous to those studied originally by Bloch, Gersten, and Quillen.  In particular, these are the resolutions originally used by Quillen \cite{QuillenHigherKTheoryI72} to prove the existence of the isomorphism \hyperref[equblochintro]{\ref{equblochintro}}. A major technical obstacle to generalizing Green and Griffiths' results is the problem of establishing the existence of such resolutions for certain singular schemes defined by ``infinitesimal thickening" of $X$.  On page 186 of their book \cite{GreenGriffithsTangentSpaces05}, Green and Griffiths pose the following question concerning this problem:
 
\begin{quote}
{\it Can one define the Bloch-Gersten-Quillen sequence $\mathcal{G}_j$ on infinitesimal neighborhoods $X_j$ of $X$
so that} 
\begin{equation}\label{equGGquestion}
 \tn{ker}(\mathcal{G}_1 \to \mathcal{G}_0) =  \underline{\underline{T}}\mathcal{G}_0?
\end{equation}
\end{quote}
We take a brief detour to explain the notation and terminology of this question.  $X_j$ denotes the infinitesimal thickening $X \times_{ \tn{Spec}(k)} \tn{Spec}\big(k[\ee]/(\ee^{j+1})\big)$ of a nonsingular quasiprojective algebraic variety $X$ of dimension $n$ over a field of characteristic zero, and $\mathcal{G}_0$ denotes the ordinary Bloch-Gersten-Quillen resolution, i.e., the Cousin resolution of the $K$-theory sheaf $\ms{K}_p$ for some $p$.  A ``satisfactory definition" ought to make sense whether or not $p=n$.     $\underline{\underline{T}}\mathcal{G}_0$ denotes the ``tangent sequence" of $\mathcal{G}_0$, with the double-underline indicating that this is a sequence of sheaves on $X$.   We do not use this notation elsewhere in the present paper.  In the case $j=1$, the question implicitly assumes that $\underline{\underline{T}}\mathcal{G}_0$ is the Cousin resolution of the tangent sheaf $T\ms{K}_{p}$.  Further specializing to the case $p=2$, $\underline{\underline{T}}\mathcal{G}_0$ may be identified with the Cousin resolution of $\Omega_{X/\ZZ}^1$, via a result of Van der Kallen \cite{VanderKallenEarlyTK271} underlying Bloch's formula \hyperref[equblochinf]{\ref{equblochinf}}.  In the present paper, we demonstrate that ``satisfactory answer" to this question requires reframing the entire picture in terms of Bass-Thomason nonconnective $K$-theory \cite{Thomason-Trobaugh90}, the negative cyclic homology of Weibel \cite{WeibelCyclicHomologySchemes91} and Keller \cite{KellerCycHomofDGAlgebras96}, \cite{KellerCyclicHomologyofExactCat96}, \cite{KellerCyclicHomologyofSchemes98}, and the relative Chern character, as treated by Corti\~{n}as, Haesemeyer, Schlichting, and Weibel \cite{WeibelCycliccdh-CohomNegativeK06}, \cite{WeibelInfCohomChernNegCyclic08}. 


\subsection{Results of This Paper.}
\label{subsecConiveau}

The purpose of this paper is to give an affirmative answer to the question \hyperref[equGGquestion]{\ref{equGGquestion}} of Green and Griffiths, while placing their concrete geometric approach in a natural and formally rigorous structural context.  From a purely $K$-theoretic perspective, this is more an exercise in {\it synthesis} than creativity; for example, the desired Bloch-Gersten-Quillen-type sequences have already appeared explicitly in the work of Colliot-Th\'el\`ene, Hoobler, and Kahn \cite{CHKBloch-Ogus-Gabber97}, and their existence follows easily from basic properties of Thomason's version of $K$-theory \cite{Thomason-Trobaugh90}. However, additional results necessary for a formally satisfactory treatment of the topic as a whole have emerged only recently.  Further, much of the geometric interest inherent in the problem comes directly from the viewpoint of Green and Griffiths, which has not previously been examined in a rigorous modern context. 
 
Our main theorem  \hyperref[maintheorem]{\ref{maintheorem}} establishes the existence of special commutative diagrams of sheaves on $X$, of the form represented schematically in diagram \hyperref[figschematicconiveau]{1.3.1} below:

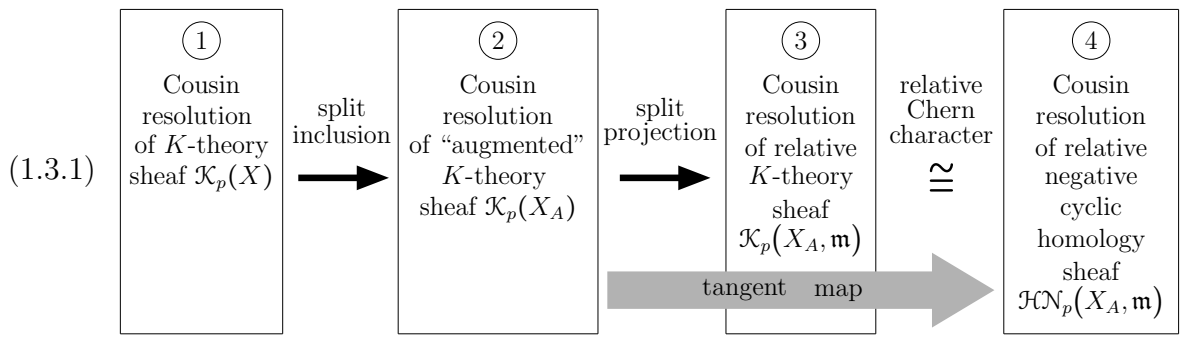
\begin{pgfpicture}{0cm}{0cm}{17cm}{3.7cm}
\pgfputat{\pgfxy(.5,1.5)}{\pgfbox[center,center]{(1.3.1)}}
\begin{pgfmagnify}{.83}{.83}
\begin{pgftranslate}{\pgfpoint{1.3cm}{-2.6cm}}
\begin{pgftranslate}{\pgfpoint{-.5cm}{0cm}}
\pgfxyline(.9,1.8)(.9,7)
\pgfxyline(.9,7)(3.5,7)
\pgfxyline(3.5,7)(3.5,1.8)
\pgfxyline(3.5,1.8)(.9,1.8)
\pgfputat{\pgfxy(2.2,5.8)}{\pgfbox[center,center]{Cousin}}
\pgfputat{\pgfxy(2.2,5.3)}{\pgfbox[center,center]{resolution }}
\pgfputat{\pgfxy(2.2,4.8)}{\pgfbox[center,center]{of $K$-theory}}
\pgfputat{\pgfxy(2.2,4.3)}{\pgfbox[center,center]{sheaf $\ms{K}_p(X)$}}
\pgfnodecircle{Node0}[stroke]{\pgfxy(2.2,6.5)}{0.3cm}
\pgfputat{\pgfxy(2.2,6.5)}{\pgfbox[center,center]{$1$}}
\end{pgftranslate}
\begin{pgftranslate}{\pgfpoint{4.05cm}{0cm}}
\pgfxyline(.8,1.8)(.8,7)
\pgfxyline(.8,7)(4,7)
\pgfxyline(4,7)(4,1.8)
\pgfxyline(4,1.8)(.8,1.8)
\pgfputat{\pgfxy(2.4,5.8)}{\pgfbox[center,center]{Cousin}}
\pgfputat{\pgfxy(2.4,5.3)}{\pgfbox[center,center]{resolution}}
\pgfputat{\pgfxy(2.4,4.8)}{\pgfbox[center,center]{of ``augmented"}}
\pgfputat{\pgfxy(2.4,4.3)}{\pgfbox[center,center]{$K$-theory }}
\pgfputat{\pgfxy(2.4,3.8)}{\pgfbox[center,center]{sheaf $\ms{K}_p(X_A)$}}
 \pgfnodecircle{Node0}[stroke]{\pgfxy(2.4,6.5)}{0.3cm}
\pgfputat{\pgfxy(2.4,6.5)}{\pgfbox[center,center]{$2$}}
\pgfsetendarrow{\pgfarrowtriangle{6pt}}
\pgfsetlinewidth{3pt}
\pgfxyline(-.8,4.3)(.4,4.3)
\pgfputat{\pgfxy(-.1,5.4)}{\pgfbox[center,center]{split}}
\pgfputat{\pgfxy(-.1,5)}{\pgfbox[center,center]{inclusion}}
\end{pgftranslate}
\begin{pgftranslate}{\pgfpoint{9.1cm}{0cm}}
\pgfxyline(1,1.8)(1,7)
\pgfxyline(1,7)(3.4,7)
\pgfxyline(3.4,7)(3.4,1.8)
\pgfxyline(3.4,1.8)(1,1.8)
\pgfputat{\pgfxy(2.2,5.8)}{\pgfbox[center,center]{Cousin}}
\pgfputat{\pgfxy(2.2,5.3)}{\pgfbox[center,center]{resolution}}
\pgfputat{\pgfxy(2.2,4.8)}{\pgfbox[center,center]{of relative}}
\pgfputat{\pgfxy(2.2,4.3)}{\pgfbox[center,center]{$K$-theory}}
\pgfputat{\pgfxy(2.2,3.8)}{\pgfbox[center,center]{sheaf}}
\pgfputat{\pgfxy(2.2,3.3)}{\pgfbox[center,center]{$\ms{K}_p\big(X_A,\mfr{m}\big)$}}
\pgfnodecircle{Node0}[stroke]{\pgfxy(2.2,6.5)}{0.3cm}
\pgfputat{\pgfxy(2.2,6.5)}{\pgfbox[center,center]{$3$}}
\pgfsetendarrow{\pgfarrowtriangle{6pt}}
\pgfsetlinewidth{3pt}
\pgfxyline(-.7,4.3)(.5,4.3)
\pgfputat{\pgfxy(-.05,5.4)}{\pgfbox[center,center]{split}}
\pgfputat{\pgfxy(-.05,5)}{\pgfbox[center,center]{projection}}
\end{pgftranslate}
\begin{pgftranslate}{\pgfpoint{13.75cm}{0cm}}
\pgfxyline(.8,1.8)(.8,7)
\pgfxyline(.8,7)(3.6,7)
\pgfxyline(3.6,7)(3.6,1.8)
\pgfxyline(3.6,1.8)(.8,1.8)
\pgfputat{\pgfxy(2.2,5.8)}{\pgfbox[center,center]{Cousin}}
\pgfputat{\pgfxy(2.2,5.3)}{\pgfbox[center,center]{resolution}}
\pgfputat{\pgfxy(2.2,4.8)}{\pgfbox[center,center]{of relative}}
\pgfputat{\pgfxy(2.2,4.3)}{\pgfbox[center,center]{negative}}
\pgfputat{\pgfxy(2.2,3.8)}{\pgfbox[center,center]{cyclic}}
\pgfputat{\pgfxy(2.2,3.3)}{\pgfbox[center,center]{homology}}
\pgfputat{\pgfxy(2.2,2.8)}{\pgfbox[center,center]{sheaf}}
\pgfputat{\pgfxy(2.2,2.3)}{\pgfbox[center,center]{$\ms{HN}_p\big(X_A,\mfr{m}\big)$}}
\pgfputat{\pgfxy(-.2,5.8)}{\pgfbox[center,center]{relative}}
\pgfputat{\pgfxy(-.2,5.4)}{\pgfbox[center,center]{Chern}}
\pgfputat{\pgfxy(-.2,5)}{\pgfbox[center,center]{character}}
\pgfputat{\pgfxy(-.2,4.3)}{\pgfbox[center,center]{\huge{$\cong$}}}
\pgfnodecircle{Node0}[stroke]{\pgfxy(2.2,6.5)}{0.3cm}
\pgfputat{\pgfxy(2.2,6.5)}{\pgfbox[center,center]{$4$}}
\end{pgftranslate}
\begin{pgftranslate}{\pgfpoint{-.2cm}{-.2cm}}
\begin{colormixin}{30!white}
\color{black}
\pgfmoveto{\pgfxy(8.4,3)}
\pgflineto{\pgfxy(13.6,3)}
\pgflineto{\pgfxy(13.6,3.3)}
\pgflineto{\pgfxy(14.6,2.7)}
\pgflineto{\pgfxy(13.6,2.1)}
\pgflineto{\pgfxy(13.6,2.4)}
\pgflineto{\pgfxy(8.4,2.4)}
\pgflineto{\pgfxy(8.4,3)}
\pgffill
\end{colormixin}
\pgfputat{\pgfxy(10.6,2.7)}{\pgfbox[center,center]{tangent}}
\pgfputat{\pgfxy(12.1,2.65)}{\pgfbox[center,center]{map}}
\end{pgftranslate}
\end{pgftranslate}
\end{pgfmagnify}
\end{pgfpicture}
\label{figschematicconiveau}

Each box in the diagram represents a complex of sheaves, arranged vertically from top to bottom, and the horizontal maps are morphisms of complexes.   $X$ is assumed to be a nonsingular quasiprojective algebraic variety over a field $k$ of characteristic zero, while $A$ is an Artinian local $k$-algebra with maximal ideal $\mfr{m}$ and residue field $k$, and $X_A := X\times_{\tn{Spec}(k)}\tn{Spec}(A)$ is the infinitesimal thickening of $X$ with respect to $A$.  The question posed by Green and Griffiths pertains to the special case where $A=k[\ee]/(\ee^{j+1})$, and the desired Bloch-Gersten-Quillen-type sequences are given by the second column in the diagram. 

In \cite{DribusDissertation}, this diagram is called the {\it coniveau machine for the thickening} $X\mapsto X_A$, since the entire construction arises from the general theory of coniveau filtration of a topological space.  To be precise, the construction in \cite{DribusDissertation} is carried out at the level of the coniveau spectral sequence, so that the resulting machine covers all indices $p$ simultaneously.  Here, for simplicity, we work in terms of a single index.  The conceptual purpose of the coniveau machine is to convert information about ``arcs of cycle classes," expressed in terms of algebraic $K$-theory, into information about their ``tangent elements," expressed in terms of negative cyclic homology.  More specifically, the composition of maps of complexes between the second and fourth columns induces maps on cohomology carrying ``arcs in $CH^p(X)$" to their ``tangents."  This approach applies more generally to the formal completions at the identity of the Chow groups, since the construction makes sense for a general Artinian local $k$-algebra $A$ with residue field $k$.   

We reproduce our main theorem here for the convenience of the reader.  Diagram \hyperref[figconiveauprecise]{4.1.1}, referenced in the statement of the theorem, is the precise version of the coniveau machine appearing schematically in diagram \hyperref[figschematicconiveau]{1.3.1} above.  The notation is explained in detail in section \hyperref[nonconnective]{4} below. 

{\bf Theorem \hyperref[maintheorem]{\ref{maintheorem}}.} {\it Let $X$ be a nonsingular quasiprojective variety of dimension $n$ over a field $k$ of characteristic zero, and let $A$ be an Artinian local $k$-algebra with maximal ideal $\mfr{m}$ and residue field $k$. Let $X_A$ be the infinitesimal thickening of $X$ with respect to $A$.  Under these conditions, there exists a commutative diagram of sheaves on $X$, as shown below.  The columns are flasque resolutions of their respective initial terms. The first three columns form a split exact sequence, and the map between the last two columns is an isomorphism of complexes induced by the relative algebraic Chern character. This diagram may be referred to as the coniveau machine for the thickening $X\mapsto X_A$.} 

\begin{pgfpicture}{0cm}{0cm}{17cm}{6.2cm}
\pgfputat{\pgfxy(.5,2.7)}{\pgfbox[center,center]{(4.1.1)}}
\begin{pgfmagnify}{.8}{.8}
\begin{pgftranslate}{\pgfpoint{4.6cm}{-4.8cm}}
\pgfputat{\pgfxy(12.5,12)}{\pgfbox[center,center]{$\ms{HN}_p(X_A,\mm)$}}
\pgfputat{\pgfxy(7.5,12)}{\pgfbox[center,center]{$\ms{K}_p(X_A,\mm)$}}
\pgfputat{\pgfxy(3,12)}{\pgfbox[center,center]{$\ms{K}_p(X_A)$}}
\pgfputat{\pgfxy(-1.2,12)}{\pgfbox[center,center]{$\ms{K}_p(X)$}}
\pgfputat{\pgfxy(12.5,10)}{\pgfbox[center,center]{$\underline{HN_p(X_A, \mm\tn{ on } \eta)}$}}
\pgfputat{\pgfxy(7.5,10)}{\pgfbox[center,center]{$\underline{K_p(X_A, \mm\tn{ on } \eta)}$}}
\pgfputat{\pgfxy(3,10)}{\pgfbox[center,center]{$\underline{K_p(X_A\tn{ on } \eta)}$}}
\pgfputat{\pgfxy(-1.2,10)}{\pgfbox[center,center]{$\underline{K_p(X\tn{ on } \eta)}$}}
\pgfputat{\pgfxy(12.5,8)}{\pgfbox[center,center]{$\displaystyle\bigoplus_{x\in X^{(1)}}\underline{HN_{p-1}(X_A,\mm\tn{ on } x)}$}}
\pgfputat{\pgfxy(7.5,8)}{\pgfbox[center,center]{$\displaystyle\bigoplus_{x\in X^{(1)}}\underline{K_{p-1}(X_A,\mm\tn{ on } x)}$}}
\pgfputat{\pgfxy(3,8)}{\pgfbox[center,center]{$\displaystyle\bigoplus_{x\in X^{(1)}}\underline{K_{p-1}(X_A\tn{ on } x)}$}}
\pgfputat{\pgfxy(-1.2,8)}{\pgfbox[center,center]{$\displaystyle\bigoplus_{x\in X^{(1)}}\underline{K_{p-1}(X\tn{ on } x)}$}}
\pgfputat{\pgfxy(10,12.3)}{\pgfbox[center,center]{$\tn{ch}$}}
\pgfputat{\pgfxy(10,11.8)}{\pgfbox[center,center]{$\sim$}}
\pgfputat{\pgfxy(5.35,12.3)}{\pgfbox[center,center]{$j$}}
\pgfputat{\pgfxy(1,12.3)}{\pgfbox[center,center]{$i$}}
\pgfputat{\pgfxy(10,10.3)}{\pgfbox[center,center]{$\tn{ch}$}}
\pgfputat{\pgfxy(10,9.8)}{\pgfbox[center,center]{$\sim$}}
\pgfputat{\pgfxy(5.35,10.3)}{\pgfbox[center,center]{$j$}}
\pgfputat{\pgfxy(1,10.3)}{\pgfbox[center,center]{$i$}}
\pgfputat{\pgfxy(10,8.45)}{\pgfbox[center,center]{$\tn{ch}$}}
\pgfputat{\pgfxy(10,7.95)}{\pgfbox[center,center]{$\sim$}}
\pgfputat{\pgfxy(5.35,8.45)}{\pgfbox[center,center]{$j$}}
\pgfputat{\pgfxy(1,8.45)}{\pgfbox[center,center]{$i$}}
\pgfputat{\pgfxy(10,5.3)}{\pgfbox[center,center]{$\tn{ch}$}}
\pgfputat{\pgfxy(10,4.8)}{\pgfbox[center,center]{$\sim$}}
\pgfputat{\pgfxy(5.35,5.3)}{\pgfbox[center,center]{$j$}}
\pgfputat{\pgfxy(1,5.3)}{\pgfbox[center,center]{$i$}}
\pgfsetendarrow{\pgfarrowpointed{3pt}}
\pgfxyline(-1.2,11.5)(-1.2,10.5)
\pgfxyline(2.75,11.5)(2.75,10.5)
\pgfxyline(7.5,11.5)(7.5,10.5)
\pgfxyline(12.5,11.5)(12.5,10.5)
\pgfxyline(-1.2,9.5)(-1.2,8.5)
\pgfxyline(2.75,9.5)(2.75,8.5)
\pgfxyline(7.5,9.5)(7.5,8.5)
\pgfxyline(12.5,9.5)(12.5,8.5)
\pgfxyline(-1.2,7.7)(-1.2,7.2)
\pgfxyline(2.75,7.7)(2.75,7.2)
\pgfxyline(7.5,7.7)(7.5,7.2)
\pgfxyline(12.5,7.7)(12.5,7.2)
\pgfxyline(8.7,12)(11.2,12)
\pgfxyline(3.9,12)(6.3,12)
\pgfxyline(-.3,12)(2,12)
\pgfxyline(9,10)(10.8,10)
\pgfxyline(4.4,10)(5.9,10)
\pgfxyline(.1,10)(1.7,10)
\pgfxyline(9.7,8.15)(10.3,8.15)
\pgfxyline(5,8.15)(5.5,8.15)
\pgfxyline(.6,8.15)(1.2,8.15)
\pgfnodecircle{Node1}[fill]{\pgfxy(-1.2,7)}{0.025cm}
\pgfnodecircle{Node1}[fill]{\pgfxy(-1.2,6.9)}{0.025cm}
\pgfnodecircle{Node1}[fill]{\pgfxy(-1.2,6.8)}{0.025cm}
\pgfnodecircle{Node1}[fill]{\pgfxy(-1.2,6.5)}{0.025cm}
\pgfnodecircle{Node1}[fill]{\pgfxy(-1.2,6.4)}{0.025cm}
\pgfnodecircle{Node1}[fill]{\pgfxy(-1.2,6.3)}{0.025cm}
\pgfnodecircle{Node1}[fill]{\pgfxy(2.75,7)}{0.025cm}
\pgfnodecircle{Node1}[fill]{\pgfxy(2.75,6.9)}{0.025cm}
\pgfnodecircle{Node1}[fill]{\pgfxy(2.75,6.8)}{0.025cm}
\pgfnodecircle{Node1}[fill]{\pgfxy(2.75,6.5)}{0.025cm}
\pgfnodecircle{Node1}[fill]{\pgfxy(2.75,6.4)}{0.025cm}
\pgfnodecircle{Node1}[fill]{\pgfxy(2.75,6.3)}{0.025cm}
\pgfnodecircle{Node1}[fill]{\pgfxy(7.5,7)}{0.025cm}
\pgfnodecircle{Node1}[fill]{\pgfxy(7.5,6.9)}{0.025cm}
\pgfnodecircle{Node1}[fill]{\pgfxy(7.5,6.8)}{0.025cm}
\pgfnodecircle{Node1}[fill]{\pgfxy(7.5,6.5)}{0.025cm}
\pgfnodecircle{Node1}[fill]{\pgfxy(7.5,6.4)}{0.025cm}
\pgfnodecircle{Node1}[fill]{\pgfxy(7.5,6.3)}{0.025cm}
\pgfnodecircle{Node1}[fill]{\pgfxy(12.5,7)}{0.025cm}
\pgfnodecircle{Node1}[fill]{\pgfxy(12.5,6.9)}{0.025cm}
\pgfnodecircle{Node1}[fill]{\pgfxy(12.5,6.8)}{0.025cm}
\pgfnodecircle{Node1}[fill]{\pgfxy(12.5,6.5)}{0.025cm}
\pgfnodecircle{Node1}[fill]{\pgfxy(12.5,6.4)}{0.025cm}
\pgfnodecircle{Node1}[fill]{\pgfxy(12.5,6.3)}{0.025cm}
\begin{pgftranslate}{\pgfpoint{0cm}{-1.2cm}}
\pgfputat{\pgfxy(12.5,6)}{\pgfbox[center,center]{$\displaystyle\bigoplus_{x\in X^{(n)}}\underline{HN_{p-n}(X_A,\mm\tn{ on } x)}$}}
\pgfputat{\pgfxy(7.5,6)}{\pgfbox[center,center]{$\displaystyle\bigoplus_{x\in X^{(n)}}\underline{K_{p-n}(X_A,\mm\tn{ on } x)}$}}
\pgfputat{\pgfxy(3,6)}{\pgfbox[center,center]{$\displaystyle\bigoplus_{x\in X^{(n)}}\underline{K_{p-n}(X_A\tn{ on } x)}$}}
\pgfputat{\pgfxy(-1.2,6)}{\pgfbox[center,center]{$\displaystyle\bigoplus_{x\in X^{(n)}}\underline{K_{p-n}(X\tn{ on } x)}$}}
\pgfsetendarrow{\pgfarrowpointed{3pt}}
\pgfxyline(-1.4,7.2)(-1.4,6.7)
\pgfxyline(2.75,7.2)(2.75,6.7)
\pgfxyline(7.5,7.2)(7.5,6.7)
\pgfxyline(12.5,7.2)(12.5,6.7)
\pgfxyline(9.7,6.15)(10.3,6.15)
\pgfxyline(5,6.15)(5.5,6.15)
\pgfxyline(.7,6.15)(1.2,6.15)
\end{pgftranslate}
\end{pgftranslate}
\end{pgfmagnify}
\end{pgfpicture}
\label{figconiveauprecise}


\subsection{Building Blocks of Our Approach.}
\label{subsecBuildingBlocks}

Our approach proceeds by systematically replacing the technical tools of Green and Griffiths with more sophisticated modern counterparts.  The {\it nonconnective $K$-theory} of Bass and Thomason subsumes the role of Milnor $K$-theory, while the {\it negative cyclic homology} of Weibel and Keller takes the place of absolute K\"{a}hler differentials.  The {\it algebraic Chern character} connects these two theories and induces a ``tangent map,"  carrying ``arcs of cycle classes" to their ``tangent elements." In more detail, our approach is built on the following results:

\begin{itemize}

\item[1.] The spectrum-valued functor $\mbf{K}$ of Bass-Thomason nonconnective $K$-theory \cite{Thomason-Trobaugh90} and the spectrum-valued functor $\mbf{HN}$ of negative cyclic homology as described by Keller \cite{KellerCycHomofDGAlgebras96}, \cite{KellerCyclicHomologyofExactCat96}, \cite{KellerCyclicHomologyofSchemes98}, are {\it substrata for cohomology theories with supports} in the sense of Colliot-Th\'{e}l\`{e}ne, Hoobler, and Kahn \cite{CHKBloch-Ogus-Gabber97}.  Here, the source category for $\mbf{K}$ and $\mbf{HN}$ is a ``suitable category of schemes," often more general than the category of nonsingular quasiprojective varieties over a field of characteristic zero.  For a suitable scheme $Y$, $\mbf{K}$ and $\mbf{HN}$ induce {\it coniveau spectral sequences}, defined in terms of the {\it cohomology groups with supports} $K_p(Y\tn{ on } Z)$ and $HN_p(Y\tn{ on } Z)$ induced by $\mbf{K}$ and $\mbf{HN}$.  The index $p$ can assume {\it negative} as well as non-negative values, and this has geometric consequences.

\item[2.] $\mbf{K}$ and $\mbf{HN}$ satisfy technical conditions called {\it \'etale excision} and the {\it projective bundle formula,} which together imply that they are {\it effaceable functors.} The {\it effacement theorem} of \cite{CHKBloch-Ogus-Gabber97} then implies exactness, and even {\it universal exactness}, of the sheafified Cousin complexes arising from the coniveau spectral sequences induced by $\mbf{K}$ and $\mbf{HN}$ for a suitable scheme $Y$.  In the case where $Y$ is the $j$th infinitesimal thickening $X_j$ of a nonsingular quasiprojective variety $X$ over a field of characteristic zero, these Cousin complexes are precisely the Bloch-Gersten-Quillen-type sequences $\mathcal{G}_j$ sought by Green and Griffiths in \hyperref[equGGquestion]{\ref{equGGquestion}}.

\item[3.]  As demonstrated by Corti\~{n}as, Haesemeyer, Schlichting, and Weibel \cite{WeibelCycliccdh-CohomNegativeK06}, \cite{WeibelInfCohomChernNegCyclic08}, 
 there exists a natural transformation of functors
\[
 \tn{ch}: \mbf{K} \to \mbf{HN},
\]
called the {\it algebraic Chern character}, which induces morphisms between the coniveau spectral sequences induced by $\mbf{K}$ and $\mbf{HN}$ for a suitable scheme $Y$.  The relative version of the algebraic Chern character \cite{WeibelInfCohomChernNegCyclic08} induces homotopy equivalences of spectra $\tn{ch}:\mbf{K}(Y, \ms{I}) \cong \mbf{HN} (Y, \ms{I})$, where $\ms{I}$ is a sheaf of nilpotent ideals on $Y$.  These equivalences may be regarded as a ``space-level Goodwillie-type results" analogous to Goodwillie's theorem at the group level \cite{Goodwillie86}.  They induce group-level isomorphisms $\tn{ch}_p: K_p(Y,\ms{I}) \cong HN_p(Y,\ms{I})$.  Similar results hold for groups with supports.

\item[4.] The groups $HN_p(Y, \ms{I})$ can often be explicitly calculated. For instance, if $Y$ is the $j$th infinitesimal thickening $X_j$ of a smooth variety $X$, these groups may be expressed in terms of absolute K\"{a}hler differentials.  Hesselholt's theorem \cite{HesselholtTruncatedPolyAlgNA} gives such a result at the level of algebras. 

\item[5.] The structural picture may be refined in terms of the {\it Adams operations} on algebraic $K$-theory and negative cyclic homology \cite{S}, \cite{Gr-1}, \cite{Gr-2}, \cite{CathelineauLambdaStructures91}, \cite{Lodayoper89}, \cite{LodayCyclicHomology98}, which induce eigenspace decompositions of the corresponding coniveau spectral sequences and related structures. In the case of algebraic $K$-theory, the ``Milnor part" of the theory may be extracted as the Adams piece of highest weight, under suitable hypotheses. This is the piece most amenable to explicit computations. 

\end{itemize}


\subsection{Organization of this Paper.}
\label{subsecOrganization.}
In section \hyperref[tandef]{2}, we discuss the general problem of defining tangent groups at the identity, and more generally, formal completions at the identity, of the Chow groups.  This requires choosing an {\it extension of the Chow functor} $CH^p$ to admit a source category including certain singular schemes; in particular, the infinitesimal thickenings $X_A$ discussed above.  We choose a $K$-theoretic definition, motivated by the fundamental isomorphism \hyperref[equblochintro]{\ref{equblochintro}}.  The specific version of $K$-theory chosen is of major significance in the subsequent analysis; we choose the nonconnective $K$-theory of Bass and Thomason, due to its convenient functorial properties.  We also review the definition used by Green and Griffiths, which involves Milnor $K$-theory, and explain why this definition neglects potentially valuable information.  

In section \hyperref[ggsurface]{3}, we briefly review Green and Griffiths' study of $TCH^2(X)$ for a nonsingular quasiprojective complex algebraic surface $X$.   Here a ``rudimentary version" of the coniveau machine appears, expressed in terms of Milnor $K$-theory and  absolute K\"{a}hler differentials.  In section \hyperref[nonconnective]{4}, we state our main theorem and explain its content.  In sections \hyperref[Kcyclicschemes]{5} and \hyperref[cohsupport]{6}, we assemble the technical elements necessary to prove our main theorem.  These are nonconnective $K$-theory, negative cyclic homology, and the relative algebraic Chern character.   In section \hyperref[sectionproof]{\ref{sectionproof}}, we give the proof.  In section \hyperref[lambda]{\ref{lambda}}, we discuss the Adams operations for algebraic $K$-theory and negative cyclic homology.  This allows the coniveau machine to be decomposed into separate parts for each Adams weight.  The construction of Green and Griffiths involves only the top-weight part; the remaining parts may lead to interesting new invariants.


\subsection{Acknowledgements.}
\label{subsecAck}
We would like to thank Marco Schlichting for his generous assistance in this investigation.  J. W. Hoffman was supported in part by NSA grant 115-60-5012 and NSF grant OISE-1318015.


\section{Tangent Groups and Formal Completions of Chow Groups}
\label{secTanDef}

\subsection{Preliminary Remarks.}
\label{subsecPrelim}

Defining the tangent group at the identity $TCH^p(X)$ of the Chow group $CH^p(X)$ is a subtle problem, involving a choice of extension of the Chow functor $CH^p$ to admit a source category including infinitesimal thickenings of smooth algebraic varieties.  We begin by discussing previous work on this subject, focusing on the contributions of Bloch, Van der Kallen, Stienstra, and Green and Griffiths.  Next, we present our definition of $TCH^p(X)$, followed by our definition of the corresponding formal completion.  We explain why the problem of definition is not straightforward, compare our definition to the definition of Green and Griffiths, and mention the desirability of a universal definition.  Throughout this section, we take $X$ to be a nonsingular quasiprojective variety over a field $k$ of characteristic zero.  


\subsection{Early Progress: Bloch, Van der Kallen, Stienstra.}
\label{subsecTangentGroups}

Spencer Bloch \cite{BlochTangentSpace72},  \cite{BlochK2Cycles74}, \cite{BlochKTheoryGeometry72}, was the principal early advocate of ``linearizing the problem" of Chow groups by studying their {\it tangent groups at the identity} via the $K$-theoretic approach.  These tangent groups are analogous to Lie algebras in an obvious sense.  In particular, the first Chow group $CH^1(X)$ of $X$ is isomorphic to the Picard group $\tn{Pic}(X)$ of $X$, an algebraic group, so the ``classical case" of codimension-$1$ cycles falls under the rubric of algebraic Lie theory.   $\tn{Pic}(X)$ had long before been identified with the sheaf cohomology group $H^1(X,\ms{O}_X^*)$; for example, in the ``classical" theory of divisors and line bundles.  Here, $\ms{O}_X^*$ is the sheaf of multiplicative groups of the structure sheaf, which is just $\ms{K}_1(X)$ from the $K$-theoretic viewpoint.  Its tangent sheaf $T\ms{O}_X^*$ is isomorphic to the structure sheaf $\ms{O}_X$ itself; this is analogous to the fact that the Lie algebra of $GL_n$ is $\mfr{gl}_n$.  Hence, 
\[
TCH^1(X)\cong TH^1\big(X, \ms{K}_1(X)\big)= H^1\big(X, T\ms{K}_1(X)\big)\cong H^1(X, \ms{O}_X).
\]
With the case $p=1$ already well-understood, Bloch turned to the general case, focusing first on the case $p=2$.  Suppressing historical details, he used the definition
\begin{equation}\label{equblochk2}
TCH^2(X):=H^2\big(X, T\ms{K}_2(X)\big).
\end{equation}
Note that replacing the $K$-theory sheaf $\ms{K}_2(X)$ with the sheaf $\ms{K}_2^{\tn{M}}(X)$ of Milnor $K$-theory does not change this definition, since the two sheaves coincide, but the analogous statements for $\ms{K}_3$, et cetera, are false!  Equation \hyperref[equblochk2]{\ref{equblochk2}} reflects a nontrivial insight. Na\"{i}vely, one might try to define tangent spaces to Chow groups by taking kernels:
\[
TCH^p(X) \overset{?}{:=} \tn{ker}\big(CH^p(X_1)\longrightarrow CH^p(X)\big),
\]
where $X_1$ is the infinitesimal thickening $X \times _{\tn{Spec} (k)} \tn{Spec} \big(k[\ee]/(\ee^2)\big)$, and where the map is induced by the canonical surjection of $k$-algebras $k[\ee]/(\ee^2) \to k$ sending $\ee$ to zero.  However, such a definition would be devoid of content, since infinitesimal thickening of a scheme does not alter its cycle groups.  In this context, the fundamental isomorphism \hyperref[equblochintro]{\ref{equblochintro}} furnishes an {\it extension of the Chow functor}, since the right-hand-side has a distinct meaning for thickened varieties such as $X_1$.  

A key ingredient of Bloch's analysis is Van der Kallen's early work on ``linearized" versions of algebraic $K$-theory.   In particular, Van der Kallen's theorem \cite{VanderKallenEarlyTK271} states that for a suitable $k$-algebra $R$, 
\begin{equation}\label{equvanderkallen}
TK_2 (R)\cong \Omega^1 _{R/\ZZ},
\end{equation}
where $\Omega ^1_{R/\ZZ}$ is the group of absolute K\"{a}hler differentials.  Here, the tangent group $TK_2(R)$ is defined in the usual way as the kernel of the map $K_p\big(R[\ee]/(\ee^2)\big)   \longrightarrow K_p(R)$.  The appearance of absolute differentials in this context already points to essential differences between the case $p=1$ and the general case.  Sheafifying Van der Kallen's theorem and substituting into equation \hyperref[equblochk2]{\ref{equblochk2}} leads to Bloch's famous formula
\[
TCH^2(X)\cong H^2(X, \Omega^1 _{X/\ZZ}),
\]
cited in equation \hyperref[equblochinf]{\ref{equblochinf}} above. 

In the middle 1980's, Jan Stienstra defined and studied the {\it formal completions} of the Chow groups, in effect replacing the algebra of dual numbers $k[\ee]/(\ee^2)$ with the category of Artinian local $k$-algebras with residue field $k$.  This seminal work \cite{Stienstra85} focuses mostly on the case where the ground field $k$ has positive characteristic, leaving the case of characteristic zero to a brief appendix. The case of positive characteristic is indeed much richer, even for codimension-$1$ cycles, where one is concerned with the formal completions at the identity of abelian varieties. Here the whole apparatus of {\it Cartier-Dieudonn\'e theory} and its generalizations comes into play.  The case of characteristic zero is already highly nontrivial, however.  An earlier paper of Stienstra \cite{Stienstra83} makes a special study of codimension-$2$ cycles on algebraic surfaces in this context.  


\subsection{Green and Griffiths.}
\label{subsecGreenGriffiths}

The recent book of Green and Griffiths \cite{GreenGriffithsTangentSpaces05} adopts a concrete geometric approach to the structure of the tangent groups at the identity $TCH^p(X)$.  Aside from a few parenthetical remarks, Green and Griffiths limit consideration to cycles of maximal codimension; i.e., zero-cycles.  The main thrust of their work is more specific still, focusing on the case of $TCH^2(X)$, where $X$ is a nonsingular complex algebraic surface. Their principal contribution is to furnish a {\it geometric interpretation} of Bloch's formula
 \[
 TCH^2(X)\cong H^2(X, \Omega ^1 _{X/\ZZ}),
 \]
 appearing in equation \hyperref[equblochinf]{\ref{equblochinf}} above.  Green and Griffiths take the viewpoint that an element of $TCH^p(X)$ should resemble a ``derivative of an arc of cycle classes."  This concrete perspective presents conceptual advantages, along with technical difficulties.  In certain special cases, a clear way forward is evident without the use of exotic modern machinery.  For example, in the special case of subvarieties, one is spared complications such as ``creation-annihilation arcs," introduced by allowing negative multiplicities, and the theory can be made precise in terms of Hilbert schemes.  In the special case of zero-cycles, arcs of cycle classes may be treated via the theory of symmetric products.  For general cycles, however, the situation is subtler. 
 
A complementary viewpoint to studying ``derivatives of arcs of cycle classes" is to investigate the ``global structure" of $TCH^p(X)$, postponing consideration of individual ``arcs."  Green and Griffiths' book \cite{GreenGriffithsTangentSpaces05} also makes progress in this direction.  In the codimension-$2$ case, the key starting point is to regard the {\it Cousin resolution} of the sheaf $\Omega^1 _{X/\ZZ}$ as the ``tangent complex" of the Cousin resolution of $\ms{K}_2(X)$.  The latter resolution is also called the {\it Bloch-Gersten-Quillen resolution} \cite{QuillenHigherKTheoryI72}, \cite{GerstenSequences72}, \cite{BalmerNiveauSS00}.   Cousin resolutions of sheaves on $X$ arise under very general hypotheses, via the {\it coniveau filtration} of $X$ \cite{HartshorneResiduesDuality66}.  Bloch-Gersten-Quillen resolutions are acyclic, and may therefore be used to compute the cohomology groups of the corresponding $K$-theory sheaves, including the second Chow group $CH^2(X)\cong H^2\big(X,\ms{K}_2(X)\big)$ of the surface $X$ in the setting of Green and Griffiths \cite{GreenGriffithsTangentSpaces05}.  Similarly, the Cousin resolution of $\Omega^1 _{X/\ZZ}$ may be used to compute the tangent group $TCH^2(X)\cong H^2(X,\Omega^1 _{X/\ZZ})$.  We review this particular case in more detail in section \ref{ggsurface} below.    Combining these ``local" and ``global" perspectives, Green and Griffiths succeed in calculating explicit deformation classes in $H^2(X, \Omega ^1 _{X/\ZZ})$.  We note in passing that these calculations are partly inspired by earlier computations of Ang\'eniol and Lejeune-Jalabert \cite{ALJ89}, involving Chern classes of complexes of vector bundles.

Green and Griffiths are explicit in their recognition that their methods invite generalizations and abstractions in several different directions.  Their book contains as many questions as answers.  The question \hyperref[equGGquestion]{\ref{equGGquestion}} concerning Bloch-Gersten-Quillen-type sequences represents only one of the open topics introduced in their final chapter. 


\subsection{Our Definition of the Tangent Groups $TCH^p(X)$.}
\label{subsecTangentGroups}

Motivated by the fundamental isomorphism
\[\label{equbloch}CH^p(X)\cong H^p\big(X,\ms{K}_p(X)\big),\]
appearing in equation \hyperref[equblochintro]{\ref{equblochintro}} above, we propose the following $K$-theoretic definition of $TCH^p(X)$:

\begin{definition}\label{defitangent} The tangent group at the identity $TCH^p(X)$ of the Chow group $CH^p(X)$ is the $p$th Zariski sheaf cohomology group 
\begin{equation}\label{equdefitangent}TCH^p(X):=TH^p\big(X,\ms{K}_p(X)\big)=H^p\big(X,T\ms{K}_p(X)\big),\end{equation}
where $T\ms{K}_p(X):=\tn{ ker}\big(\ms{K}_p(X_1)\rightarrow\ms{K}_p(X)\big)$ is the tangent sheaf at the identity of the $K$-theory sheaf $\ms{K}_p(X)$.  
\end{definition}
The definition $T\ms{K}_p(X)$ follows the ``usual definition of a tangent functor;" see, for instance, \cite{BlochTangentSpace72} page 205.  Recall that 
$X_1$ denotes the first-order thickening $X\times_{Spec(k)}\tn{Spec} \big(k[\ee]/\ee^2\big)$, and that the map $\ms{K}_p(X_1)\rightarrow\ms{K}_p(X)$ is induced by the canonical surjection $k[\ee]/(\ee^2)\rightarrow k$ sending $\ee$ to zero.  The second equals sign in equation \hyperref[equdefitangent]{\ref{equdefitangent}}, which is equivalent to the statement that ``the tangent operation commutes with passage to cohomology," is easily justified by the fact that $H^p$ is a middle-exact functor.  Note that there exists a canonical splitting
\[
\ms{K}_p(X_1) = \ms{K}_p(X)\oplus T\ms{K}_p(X),
\]
resulting from fact that the natural surjection $\ms{O}_{X_1} \to \ms{O}_{X} $ is split by the injection $\ms{O}_{X} \to \ms{O}_{X_1}$.  This splitting, suitably generalized, is what underlies the split-exactness of the first three columns of the coniveau machine. 


\subsection{Generalization: Formal Completions.}
\label{subsecFormalCompletions}

More generally, let $\textsf{Art}_k$ be the category of Artinian local $k$-algebras with residue field $k$.  An object $A$ of $\textsf{Art}_k$ comes equipped with a natural $k$-augmentation $A \to k$ defined by sending nilpotent elements to zero, and this augmentation induces a map $\ms{K}_p(X_A)\rightarrow\ms{K}_p(X)$  of sheaves on $X$.  Following Stienstra \cite{Stienstra83}, we can use this data to define formal completions at the identity of Chow groups: 

\begin{definition}\label{defiformalcompletion} The formal completion at the identity $\widehat{CH}^p (X)$ of the Chow group $CH^p(X)$ is the functor
\begin{equation}
\widehat{CH}^p (X): \textsf{Art}_k \longrightarrow \textsf{Mod}_k, \quad
\widehat{CH}^p (X)(A) :=H^p\Big (X, \tn{ ker}\big(\ms{K}_p(X_A)\rightarrow\ms{K}_p(X)\big)\Big),
\end{equation}
where $\textsf{Mod}_k$ is the category of $k$-modules. 
\end{definition} 

Since the ``absolute" sheaf cohomology group $H^p\big(X,\ms{K}_p(X)\big)\cong CH^p(X)$ and the ``relative" sheaf cohomology group $H^p\Big (X, \tn{ ker}\big(\ms{K}_p(X_A)\rightarrow\ms{K}_p(X)\big)\Big):=\widehat{CH}^p (X)(A)$ play such a central role in our study, it is natural to ask if the ``augmented" sheaf cohomology group $H^p\big(X,\ms{K}_p(X_A)\big)$ is important as well.  The answer is yes; this group is the correct formal version of the ``group of arcs through the identity in $CH^p(X)$,"  from the concrete geometric viewpoint of Green and Griffiths.  


\subsection{Subtleties of Definition.}
\label{subsectionsubtleties}
As mentioned in section \hyperref[subsecTangentGroups]{\ref{subsecTangentGroups}}, it is futile to attempt a na\"{i}ve cycle-theoretic definition of tangent groups at the identity and formal completions at the identity of Chow groups, since infinitesimal thickenings are invisible from a topological perspective.  Instead, one must extend the Chow functor so as to capture at least some of the information involved in the thickening procedure.  This involves a nontrivial choice.  In contrast to our definition, Green and Griffiths use a different extension of the Chow functor, induced by the isomorphism
\begin{equation}\label{milnorformula}
CH^p(X)  = H^p\big(X,\ms{K}^{\tn{M}}_p(X)\big),
\end{equation}
where the superscript ``M" denotes Milnor $K$-theory.  When $p=2$, this definition is ``automatic," because $\ms{K}^{\tn{M}}_2=\ms{K}_2$.  For $p>2$, Milnor $K$-theory is only part of the story, but a recent theorem of Kerz \cite{Kerz1} ensures that the ``difference disappears upon passage to cohomology," so that equation \hyperref[milnorformula]{\ref{milnorformula}} still holds.   However, this Milnor-theoretic extension of the Chow functor produces very different results than our choice of extension when one begins to {\it deform} the picture, even infinitesimally.  To illustrate the difference, it will suffice to consider tangent groups at the identity.  Using Kerz's result, one {\it could} make the definition: 
\[
T_{\tn{GG}}CH^p (X) :=H^p\Big ( X, \tn{ ker}\big(\ms{K}^{\tn{M}}_p(X_1)\rightarrow\ms{K}^{\tn{M}}_p(X) \big)\Big),
\]
where the subscript ``GG" stands for ``Green and Griffiths."  This definition has the virtue that Milnor $K$-theory is much more accessible than general $K$-theory from a computational perspective.  However, it potentially neglects information about deformation of cycle classes, for the following subtle reason:  the validity of equation \hyperref[milnorformula]{\ref{milnorformula}}, which involves the $K$-theory sheaves $\ms{K}_p(X)$ and $\ms{K}^{\tn{M}}_p(X)$ themselves, does {\it not} imply that the cohomology groups $H^p\big(X,T\ms{K}_p(X)\big)$ and $H^p\big(X,T\ms{K}^{\tn{M}}_p(X)\big)$ of the corresponding tangent sheaves $T\ms{K}_p(X)$ and $T\ms{K}^{\tn{M}}_p(X)$ are isomorphic.  For example, it is easily shown that 
\begin{equation}\label{equTKMTK} T\ms{K}_3^{\tn{M}}(X)\cong\varOmega^2_{X/\QQ}\tn{\hspace*{.5cm} but \hspace*{.5cm}} T\ms{K}_3(X)\cong\varOmega^2_{X/\QQ}\oplus\ms{O}_X.\end{equation}
Even for deformations of zero-cycles on a $3$-fold $X$, there may be new information arising from the second factor in $T\ms{K}_3(X)$ in equation \hyperref[equTKMTK]{\ref{equTKMTK}}.  Our preference is to work with the full $K$-functor; the ``Milnor part" may be extracted {\it a posteriori} as the piece of highest Adams weight. 


\subsection{Desirability of Universal Extensions of Chow Functors.}
\label{subsectionUniversal}
Our definitions \hyperref[defitangent]{\ref{defitangent}} and \hyperref[defiformalcompletion]{\ref{defiformalcompletion}} of tangent groups at the identity and formal completions at the identity of Chow groups have many attractive features, deriving from the advantages of extending the Chow functors $CH^p$ in terms of modern $K$-theory.  However, other choices of extensions are possible, and we cannot claim that our choice is necessarily the ``best."  To investigate this topic more thoroughly, one would wish to compare extensions of $CH^p$ in a systematic way, and quantify how these extensions organize information about deformations of cycle classes.   For example, maps between various versions of $K$-theory, which exist under suitable hypotheses, permit comparison of $K$-theoretic extensions of $CH^p$, and these comparisons suggest that our definitions are ``better" than their Milnor-theoretic counterparts.  Ultimately, however, one would wish to identify extensions satisfying universal properties justifying them as the ``best" choices for studying deformations of cycle classes.  We leave this as an open topic.  


\section{$TCH^2$ of a Surface According to Green and Griffiths.}
\label{ggsurface}

\subsection{Tangent Map for Milnor $K$-Theory.}
\label{tangentmapMilnor}

Here, we summarize a special case examined in detail by Green and Griffiths \cite{GreenGriffithsTangentSpaces05}; namely, the case of $TCH^2(X)$ of a nonsingular complex algebraic surface.  This group may be described via Bloch's formula \hyperref[equblochinf]{\ref{equblochinf}} in terms of the $K$-theory sheaf $\ms{K}_2(X)$, which coincides with the corresponding sheaf $\ms{K}_2^{\tn{M}}(X)$ of Milnor $K$-theory, as noted in section \hyperref[subsectionsubtleties]{\ref{subsectionsubtleties}} above.  Hence, the description of $TCH^2(X)$ may be reduced to expressions involving the generators of Milnor $K$-theory, called {\it Steinberg symbols.} 

The canonical split surjection $k[\ee]/(\ee^2)\rightarrow k$ from the algebra of dual numbers onto the base field induces the following split exact sequence: 
\begin{equation}0\rightarrow \ms{K}_2^{\tn{M}}(X)\overset{i}{\rightarrow}\ms{K}_2^{\tn{M}}(X_1)\overset{j}{\rightarrow} T\ms{K}_2^{\tn{M}}(X)\rightarrow 0,\end{equation}
where we again remind the reader that $X_1$ denotes the first-order thickening $X\times_{Spec(k)}\tn{Spec} \big(k[\ee]/\ee^2\big)$.  By Van der Kallen's theorem \cite{VanderKallenEarlyTK271}, there exists an isomorphism $T\ms{K}_2^{\tn{M}}(X)\cong\varOmega^1_{X/\QQ}$.  This isomorphism is a special case of the relative algebraic Chern character, so we denote it by $\tn{ch}$.  Essentially, it is a ``logarithmic map;" we describe this in more detail below.  The above sequence may then be extended to a complex:
\begin{equation}\label{toprow}0\rightarrow \ms{K}_2^{\tn{M}}(X)\overset{i}{\rightarrow}\ms{K}_2^{\tn{M}}(X_1)\overset{j}{\rightarrow} T\ms{K}_2^{\tn{M}}(X)\overset{\tn{ch}}{\rightarrow} \varOmega^1_{X/\QQ}\rightarrow 0.\end{equation}
In our terminology, this complex is the top row of the coniveau machine for the thickening $X\mapsto X_1$.  The composition $T=\tn{ch}\circ j$ is 
called the {\it tangent map}; it extends to the tangent map of complexes appearing in the schematic diagram \hyperref[figschematicconiveau]{1.3.1} of the coniveau machine in section \hyperref[subsecConiveau]{\ref{subsecConiveau}} above.  


\subsection{Tangent Map in Terms of Steinberg Symbols.}
\label{tangentmapSteinberg}

The tangent map $T$ is simple enough in this case to describe explicitly.  Let 
\begin{equation}\label{equsteinberg1}\{f_0+f_1\ee,g_0+g_1\ee\},\end{equation}
 be a {Steinberg symbol representing an element of $\ms{K}_2^{\tn{M}}(X_1)(U)$ for some open set $U\subset X$.  Note that since $X$ and $X_1$ share the same Zariski topological space, $U$ may be regarded as a subset of $X$.  Projection to the kernel $T\ms{K}_2^{\tn{M}}(X)=\tn{ker}\big(\ms{K}_2^{\tn{M}}(X_1)\rightarrow \ms{K}_2^{\tn{M}}(X)\big)$ ``peels off" the constant symbol $\{f_0,g_0\}$, leaving the product
 \begin{equation}\label{equsteinberg2}\Big\{f_0,1+\frac{g_1}{g_0}\ee\Big\}\Big\{1+\frac{f_1}{f_0}\ee,g_0\Big\}\Big\{1+\frac{f_1}{f_0}\ee,1+\frac{g_1}{g_0}\ee\Big\}.\end{equation}
 Each symbol in the product in equation \hyperref[equsteinberg2]{\ref{equsteinberg2}} may be viewed as an ``infinitesimal arc through the identity in $\ms{K}_2^{\tn{M}}(X_1)(U)$," in the sense that replacing $\ee$ with zero yields trivial symbols.  Applying the relative algebraic Chern character $\tn{ch}$ then yields the differential
 \begin{equation}\label{equsteinberg3}\frac{g_1}{g_0}df_0-\frac{f_1}{f_0}dg_0\hspace*{.2cm}\in\hspace*{.2cm} \varOmega^1_{X/\QQ}(U).\end{equation}
 Passage from \hyperref[equsteinberg1]{\ref{equsteinberg1}} to \hyperref[equsteinberg3]{\ref{equsteinberg3}} is a special case of the logarithm formula of Maazen and Stienstra \cite{MaazenStienstra77}:
 \begin{equation}\{a,b\}\mapsto\frac{1}{a}\log(b) da,\end{equation}
which remains valid when the algebra of dual numbers is replaced with an arbitrary object of $\textsf{Art}_k$.


\subsection{Green and Griffiths' Version of the Coniveau Machine.}
 \label{subsecGGconiveau}
  
 Telescoping the last two terms of the complex \hyperref[toprow]{\ref{toprow}} above via the composition $T=\tn{ch}\circ j$ yields the split exact sequence
 \begin{equation}\label{toprow2}0\rightarrow \ms{K}_2^{\tn{M}}(X)\overset{i}{\rightarrow}\ms{K}_2^{\tn{M}}(X_1)\overset{T}{\rightarrow} \varOmega^1_{X/\QQ}\rightarrow 0.\end{equation}
Green and Griffiths extend this sequence to obtain diagram \hyperref[figGGconiveau]{3.3.2} below.  More precisely, this diagram may be {\it inferred} in its entirely from Green and Griffiths \cite{GreenGriffithsTangentSpaces05}, though only pieces of it appear explicitly there. 

 \begin{pgfpicture}{0cm}{0cm}{17cm}{7.5cm}
\pgfputat{\pgfxy(0,3.75)}{\pgfbox[left,center]{(3.3.2)}}
\begin{pgftranslate}{\pgfpoint{.5cm}{-5cm}}
\pgfputat{\pgfxy(12,12)}{\pgfbox[center,center]{$\varOmega^1_{X/\QQ}$}}
\pgfputat{\pgfxy(7.5,12)}{\pgfbox[center,center]{$\ms{K}_2^{\tn{M}}\big(X_1)$}}
\pgfputat{\pgfxy(3,12)}{\pgfbox[center,center]{$\ms{K}_2^{\tn{M}}(X)$}}
\pgfputat{\pgfxy(12,10)}{\pgfbox[center,center]{$\underline{H_\eta^0(\varOmega^1_{X/\QQ})}$}}
\pgfputat{\pgfxy(7.5,10)}{\pgfbox[center,center]{$\displaystyle\frac{\tn{arcs}\big(\underline{K_2^{\tn{M}}\big(k(\eta)}\big)\big)}{\sim_1}$}}
\pgfputat{\pgfxy(3,10)}{\pgfbox[center,center]{$\underline{K_2^{\tn{M}}\big(k(\eta)\big)}$}}
\pgfputat{\pgfxy(12,8)}{\pgfbox[center,center]{$\displaystyle\bigoplus_{x\in X^{(1)}}\underline{H^1_x(\varOmega^1_{X/\QQ})}$}}
\pgfputat{\pgfxy(7.5,8)}{\pgfbox[center,center]{$\displaystyle\frac{\tn{arcs}\Big(\displaystyle\bigoplus_{x\in X^{(1)}}\underline{K_1^{\tn{M}}\big(k(x)}\big)\Big)}{\displaystyle\sim_1}$}}
\pgfputat{\pgfxy(3,8)}{\pgfbox[center,center]{$\displaystyle\bigoplus_{x\in X^{(1)}}\underline{K_1^{\tn{M}}\big(k(x)\big)}$}}
\pgfputat{\pgfxy(12,6)}{\pgfbox[center,center]{$\displaystyle\bigoplus_{x\in X^{(2)}}\underline{H^2_x(\varOmega^1_{X/\QQ})}$}}
\pgfputat{\pgfxy(7.5,6)}{\pgfbox[center,center]{$\displaystyle\frac{\tn{arcs}\Big(\displaystyle\bigoplus_{x\in X^{(2)}}\underline{K_0^{\tn{M}}\big(k(x)}\big)\Big)}{\displaystyle\sim_1}$}}
\pgfputat{\pgfxy(3,6)}{\pgfbox[center,center]{$\displaystyle\bigoplus_{x\in X^{(2)}}\underline{K_0^{\tn{M}}\big(k(x)\big)}$}}
\pgfsetendarrow{\pgfarrowpointed{3pt}}
\pgfxyline(3,11.5)(3,10.5)
\pgfxyline(3,9.5)(3,8.5)
\pgfxyline(3,7.7)(3,6.7)
\pgfxyline(12,11.5)(12,10.5)
\pgfxyline(12,9.5)(12,8.5)
\pgfxyline(12,7.7)(12,6.7)
\pgfxyline(8.5,12)(11,12)
\pgfxyline(3.75,12)(6.5,12)
\pgfxyline(7.5,11.5)(7.5,10.7)
\pgfxyline(7.5,9.3)(7.5,8.7)
\pgfxyline(7.5,7.2)(7.5,6.7)
\pgfxyline(9.5,10)(10.7,10)
\pgfxyline(4.3,10)(5.5,10)
\pgfxyline(9.8,8)(10.4,8)
\pgfxyline(4.6,8)(5.2,8)
\pgfxyline(9.8,6)(10.4,6)
\pgfxyline(4.6,6)(5.2,6)
\pgfputat{\pgfxy(10,12.3)}{\pgfbox[center,center]{\small{$T$}}}
\pgfputat{\pgfxy(5.2,12.3)}{\pgfbox[center,center]{\small{$i$}}}
\end{pgftranslate}
\end{pgfpicture}
\label{figGGconiveau}

Diagram \hyperref[figGGconiveau]{3.3.2} may be viewed as a ``rudimentary version of the coniveau machine."  Comparing this diagram to the schematic diagram \hyperref[figschematicconiveau]{1.3.1} above, the principal differences are that the third and fourth columns of \hyperref[figschematicconiveau]{1.3.1} are telescoped into a single column in \hyperref[figGGconiveau]{3.3.2}, $K$-theory is replaced with Milnor $K$-theory, and the second column of \hyperref[figGGconiveau]{3.3.2} is expressed only informally, in terms of ``arc objects."  

We pause to examine diagram \hyperref[figGGconiveau]{3.3.2} in more detail.  The left column is the Cousin resolution of $\ms{K}_2^{\tn{M}}(X)$, which is the familiar Bloch-Gersten-Quillen resolution.  The right column is the Cousin resolution of $\Omega _X ^1$.  Here, $\eta$ is the generic point of $X$.  For each $q\ge0$, $\underline{K_{2-q}^{\tn{M}}\big(k(x)\big)}$ denotes the skyscraper sheaf at a codimension-$q$ point $x$ of $X$ corresponding to the group $K_{2-q}^{\tn{M}}\big(k(x)\big)$, where $k(x)$ is the residue field of $x$.  Similarly,  $\underline{H_x^{2-q}(\varOmega^1_{X/\QQ})}$ denotes the skyscraper sheaf at $x$ corresponding to the local cohomology group $H_x^{2-q}(\varOmega^1_{X/\QQ})$.  The term $\underline{K_2^{\tn{M}}\big(k(\eta)\big)}$ is the constant sheaf on $X$ corresponding to the group $K_2^{\tn{M}}(k(\eta))$, where $k(\eta)=k(X)$ is the rational function field of $X$, and the term $\underline{H_\eta^0(\varOmega^1_{X/\QQ})}$ is the constant sheaf on $X$ corresponding to the  module $\Omega^1_{k(X)/\QQ}$.  Both the right and left columns are flasque resolutions of their respective initial terms.  Note that for Cousin complexes, one ordinarily proves exactness for {\it locally free sheaves of finite rank;} the sheaf $\Omega _X ^1$ of absolute K\"{a}hler differentials is not locally free of finite rank, but is a filtered inductive limit of such sheaves.  

The middle column in diagram \hyperref[figGGconiveau]{3.3.2} is only implicitly defined in the work of Green and Griffiths.  Intuitively, the ``arcs'' appearing here represent deformations, and the equivalence relation $\sim_1$ in the ``denominators" indicates that these ``arcs"  are to be considered ``modulo first-order equivalence."  Heuristically, a ``local element" of the arc object $\displaystyle\tn{arcs}\big(\underline{K_2^{\tn{M}}\big(k(\eta)}\big)\big)$ may be expressed as a product of ``variable Steinberg symbols" of the form $\{f_0+f_1\ee,g_0+g_1\ee\}$.  Imposing first-order equivalence $\sim_1$ simply means taking $\ee^2=0$.  A ``local element" of the arc object $\tn{arcs}\Big(\displaystyle\bigoplus_{x\in X^{(1)}}\underline{K_1^{\tn{M}}\big(k(x)\big)}\Big)$ may be represented by a formal sum of expressions $\big(\tn{div}(f_0+f_{1}\ee),g_0+g_{1}\ee\big)$, where $\tn{div}(f_0)$ is a local expression for a divisor $Z$ on $X$, and where $g_0 \in k(Z)^*$. We think of $\tn{div}(f_0+\ee f_{1})$ as a local first-order deformation of $Y$, and $g_0+\ee g_{1}$ as a deformation of $g_0$. A ``local element" of the arc object $\tn{arcs}\Big(\displaystyle\bigoplus_{x\in X^{(0)}}\underline{K_0^{\tn{M}}\big(k(x)\big)}\Big)$ is an ``arc of zero-cycles" $Z_{\ee}=V(f_0+f_1\ee,g+g_1\ee)$, where $V(-,-)$ denotes the common vanishing locus of two functions, and where $Z=V(f_0,g_0)$ is supported on a zero-dimensional point $x$ of $X$, since we are working locally.  

Green and Griffiths show how to calculate the ``tangent elements" of such ``arcs," and prove that the above diagram commutes. They also show that these tangent elements ``fill out their target groups" in a suitable sense.  In definition \hyperref[defigendefgroupChow]{\ref{defigendefgroupChow}} below, we give a rigorous definition of these arc objects.  This definition is obtained in the process of showing that the coniveau machine exists under much more general circumstances.   


\section{Our Main Theorem.}
\label{sectiontheorem}

\subsection{Statement of the Theorem.}
 \label{subsecGGconiveau}
 
In this section, we restate our main theorem, and explain its notation and content.   The technical machinery referenced in the theorem is developed in subsequent sections.

\begin{theorem}\label{maintheorem} Let $X$ be a nonsingular quasiprojective variety of dimension $n$ over a field $k$ of characteristic zero, and let $A$ be an Artinian local $k$-algebra with maximal ideal $\mfr{m}$ and residue field $k$. Let $X_A$ be the infinitesimal thickening of $X$ with respect to $A$.  Under these conditions, there exists a commutative diagram of sheaves on $X$, as shown below.  The columns are flasque resolutions of their respective initial terms. The first three columns form a split exact sequence, and the map between the last two columns is an isomorphism of complexes induced by the relative algebraic Chern character. This diagram may be referred to as the coniveau machine for the thickening $X\mapsto X_A$. 

\begin{pgfpicture}{0cm}{0cm}{17cm}{6.2cm}
\pgfputat{\pgfxy(.5,2.7)}{\pgfbox[center,center]{(4.1.1)}} 
\begin{pgfmagnify}{.8}{.8}
\begin{pgftranslate}{\pgfpoint{4.6cm}{-4.8cm}}
\pgfputat{\pgfxy(12.5,12)}{\pgfbox[center,center]{$\ms{HN}_p(X_A,\mm)$}}
\pgfputat{\pgfxy(7.5,12)}{\pgfbox[center,center]{$\ms{K}_p(X_A,\mm)$}}
\pgfputat{\pgfxy(3,12)}{\pgfbox[center,center]{$\ms{K}_p(X_A)$}}
\pgfputat{\pgfxy(-1.2,12)}{\pgfbox[center,center]{$\ms{K}_p(X)$}}
\pgfputat{\pgfxy(12.5,10)}{\pgfbox[center,center]{$\underline{HN_p(X_A, \mm\tn{ on } \eta)}$}}
\pgfputat{\pgfxy(7.5,10)}{\pgfbox[center,center]{$\underline{K_p(X_A, \mm\tn{ on } \eta)}$}}
\pgfputat{\pgfxy(3,10)}{\pgfbox[center,center]{$\underline{K_p(X_A\tn{ on } \eta)}$}}
\pgfputat{\pgfxy(-1.2,10)}{\pgfbox[center,center]{$\underline{K_p(X\tn{ on } \eta)}$}}
\pgfputat{\pgfxy(12.5,8)}{\pgfbox[center,center]{$\displaystyle\bigoplus_{x\in X^{(1)}}\underline{HN_{p-1}(X_A,\mm\tn{ on } x)}$}}
\pgfputat{\pgfxy(7.5,8)}{\pgfbox[center,center]{$\displaystyle\bigoplus_{x\in X^{(1)}}\underline{K_{p-1}(X_A,\mm\tn{ on } x)}$}}
\pgfputat{\pgfxy(3,8)}{\pgfbox[center,center]{$\displaystyle\bigoplus_{x\in X^{(1)}}\underline{K_{p-1}(X_A\tn{ on } x)}$}}
\pgfputat{\pgfxy(-1.2,8)}{\pgfbox[center,center]{$\displaystyle\bigoplus_{x\in X^{(1)}}\underline{K_{p-1}(X\tn{ on } x)}$}}
\pgfputat{\pgfxy(10,12.3)}{\pgfbox[center,center]{$\tn{ch}$}}
\pgfputat{\pgfxy(10,11.8)}{\pgfbox[center,center]{$\sim$}}
\pgfputat{\pgfxy(5.35,12.3)}{\pgfbox[center,center]{$j$}}
\pgfputat{\pgfxy(1,12.3)}{\pgfbox[center,center]{$i$}}
\pgfputat{\pgfxy(10,10.3)}{\pgfbox[center,center]{$\tn{ch}$}}
\pgfputat{\pgfxy(10,9.8)}{\pgfbox[center,center]{$\sim$}}
\pgfputat{\pgfxy(5.35,10.3)}{\pgfbox[center,center]{$j$}}
\pgfputat{\pgfxy(1,10.3)}{\pgfbox[center,center]{$i$}}
\pgfputat{\pgfxy(10,8.45)}{\pgfbox[center,center]{$\tn{ch}$}}
\pgfputat{\pgfxy(10,7.95)}{\pgfbox[center,center]{$\sim$}}
\pgfputat{\pgfxy(5.35,8.45)}{\pgfbox[center,center]{$j$}}
\pgfputat{\pgfxy(1,8.45)}{\pgfbox[center,center]{$i$}}
\pgfputat{\pgfxy(10,5.3)}{\pgfbox[center,center]{$\tn{ch}$}}
\pgfputat{\pgfxy(10,4.8)}{\pgfbox[center,center]{$\sim$}}
\pgfputat{\pgfxy(5.35,5.3)}{\pgfbox[center,center]{$j$}}
\pgfputat{\pgfxy(1,5.3)}{\pgfbox[center,center]{$i$}}
\pgfsetendarrow{\pgfarrowpointed{3pt}}
\pgfxyline(-1.2,11.5)(-1.2,10.5)
\pgfxyline(2.75,11.5)(2.75,10.5)
\pgfxyline(7.5,11.5)(7.5,10.5)
\pgfxyline(12.5,11.5)(12.5,10.5)
\pgfxyline(-1.2,9.5)(-1.2,8.5)
\pgfxyline(2.75,9.5)(2.75,8.5)
\pgfxyline(7.5,9.5)(7.5,8.5)
\pgfxyline(12.5,9.5)(12.5,8.5)
\pgfxyline(-1.2,7.7)(-1.2,7.2)
\pgfxyline(2.75,7.7)(2.75,7.2)
\pgfxyline(7.5,7.7)(7.5,7.2)
\pgfxyline(12.5,7.7)(12.5,7.2)
\pgfxyline(8.7,12)(11.2,12)
\pgfxyline(3.9,12)(6.3,12)
\pgfxyline(-.3,12)(2,12)
\pgfxyline(9,10)(10.8,10)
\pgfxyline(4.4,10)(5.9,10)
\pgfxyline(.1,10)(1.7,10)
\pgfxyline(9.7,8.15)(10.3,8.15)
\pgfxyline(5,8.15)(5.5,8.15)
\pgfxyline(.6,8.15)(1.2,8.15)
\pgfnodecircle{Node1}[fill]{\pgfxy(-1.2,7)}{0.025cm}
\pgfnodecircle{Node1}[fill]{\pgfxy(-1.2,6.9)}{0.025cm}
\pgfnodecircle{Node1}[fill]{\pgfxy(-1.2,6.8)}{0.025cm}
\pgfnodecircle{Node1}[fill]{\pgfxy(-1.2,6.5)}{0.025cm}
\pgfnodecircle{Node1}[fill]{\pgfxy(-1.2,6.4)}{0.025cm}
\pgfnodecircle{Node1}[fill]{\pgfxy(-1.2,6.3)}{0.025cm}
\pgfnodecircle{Node1}[fill]{\pgfxy(2.75,7)}{0.025cm}
\pgfnodecircle{Node1}[fill]{\pgfxy(2.75,6.9)}{0.025cm}
\pgfnodecircle{Node1}[fill]{\pgfxy(2.75,6.8)}{0.025cm}
\pgfnodecircle{Node1}[fill]{\pgfxy(2.75,6.5)}{0.025cm}
\pgfnodecircle{Node1}[fill]{\pgfxy(2.75,6.4)}{0.025cm}
\pgfnodecircle{Node1}[fill]{\pgfxy(2.75,6.3)}{0.025cm}
\pgfnodecircle{Node1}[fill]{\pgfxy(7.5,7)}{0.025cm}
\pgfnodecircle{Node1}[fill]{\pgfxy(7.5,6.9)}{0.025cm}
\pgfnodecircle{Node1}[fill]{\pgfxy(7.5,6.8)}{0.025cm}
\pgfnodecircle{Node1}[fill]{\pgfxy(7.5,6.5)}{0.025cm}
\pgfnodecircle{Node1}[fill]{\pgfxy(7.5,6.4)}{0.025cm}
\pgfnodecircle{Node1}[fill]{\pgfxy(7.5,6.3)}{0.025cm}
\pgfnodecircle{Node1}[fill]{\pgfxy(12.5,7)}{0.025cm}
\pgfnodecircle{Node1}[fill]{\pgfxy(12.5,6.9)}{0.025cm}
\pgfnodecircle{Node1}[fill]{\pgfxy(12.5,6.8)}{0.025cm}
\pgfnodecircle{Node1}[fill]{\pgfxy(12.5,6.5)}{0.025cm}
\pgfnodecircle{Node1}[fill]{\pgfxy(12.5,6.4)}{0.025cm}
\pgfnodecircle{Node1}[fill]{\pgfxy(12.5,6.3)}{0.025cm}
\begin{pgftranslate}{\pgfpoint{0cm}{-1.2cm}}
\pgfputat{\pgfxy(12.5,6)}{\pgfbox[center,center]{$\displaystyle\bigoplus_{x\in X^{(n)}}\underline{HN_{p-n}(X_A,\mm\tn{ on } x)}$}}
\pgfputat{\pgfxy(7.5,6)}{\pgfbox[center,center]{$\displaystyle\bigoplus_{x\in X^{(n)}}\underline{K_{p-n}(X_A,\mm\tn{ on } x)}$}}
\pgfputat{\pgfxy(3,6)}{\pgfbox[center,center]{$\displaystyle\bigoplus_{x\in X^{(n)}}\underline{K_{p-n}(X_A\tn{ on } x)}$}}
\pgfputat{\pgfxy(-1.2,6)}{\pgfbox[center,center]{$\displaystyle\bigoplus_{x\in X^{(n)}}\underline{K_{p-n}(X\tn{ on } x)}$}}
\pgfsetendarrow{\pgfarrowpointed{3pt}}
\pgfxyline(-1.4,7.2)(-1.4,6.7)
\pgfxyline(2.75,7.2)(2.75,6.7)
\pgfxyline(7.5,7.2)(7.5,6.7)
\pgfxyline(12.5,7.2)(12.5,6.7)
\pgfxyline(9.7,6.15)(10.3,6.15)
\pgfxyline(5,6.15)(5.5,6.15)
\pgfxyline(.7,6.15)(1.2,6.15)
\end{pgftranslate}
\end{pgftranslate}
\end{pgfmagnify}
\end{pgfpicture}
\label{figconiveauprecise2}
\end{theorem}


\subsection{Notation and Content.}
 \label{subsecGGconiveau}

We pause here to interpret the statement of theorem \hyperref[maintheorem]{\ref{maintheorem}} in more detail, with particular attention to diagram\hyperref[figconiveauprecise2]{4.1.1}.  All the objects in the diagram are local in the sense that they depend only on the corresponding local rings. $X^{(d)}$ denotes the set of points of codimension $d$ in the scheme $X$. For example, the expression $\underline{K_{p-d}(X\tn{ on } x)}$, appearing in the first column on the left, denotes the skyscraper sheaf on the closed set $\overline {\{ x \}}$ corresponding to the ``absolute $K$-theory group supported at $x$;" i.e., $K_{p-d}(X\tn{ on } x)$.  This group is defined as a direct limit in section \hyperref[sectioncohomsupports]{\ref{sectioncohomsupports}} below. Since $X$ is nonsingular, this group is isomorphic to $K_{p-d}\big(k(x)\big)$ by Quillen's {\it d\'evissage} theorem \cite{QuillenHigherKTheoryI72}, where $k(x)$ is the residue field of the local ring $\ms{O}_{X, x}$.   Hence, the first column of diagram \hyperref[figconiveauprecise2]{4.1.1} is just the familiar Bloch-Gersten-Quillen resolution of the sheaf $\ms{K}_p(X)$; in particular, the left-hand column of diagram \hyperref[figGGconiveau]{3.3.2} represents the special case of $\ms{K}_2(X)$, where $X$ is a smooth algebraic surface.  The groups $K_{p-d}(X\tn{ on } x)$ vanish when $p<d$.  

The groups $K_{p-d}(X_A\tn{ on } x)$, called the {\it augmented $K$-theory groups supported at $x$,} play an analogous role in the second column.  These groups need not vanish in negative degrees, but nonetheless have a relatively simple form in this case, since their ``absolute parts" vanish.  In particular, when $p<d$, $K_{p-d}(X_A\tn{ on } x)$ is isomorphic to the corresponding groups $K_{p-d}(X_A,\mm\tn{ on } x)$, called the {\it relative $K$-theory group supported at $x$,} which appear in the third column.   These groups, in turn, are isomorphic via the relative Chern character to the groups $HN_{p-d}(X_A,\mm\tn{ on } x)$, called the {\it relative cyclic homology groups supported at $x$.}   We discuss these relationships in more detail in section \hyperref[sectionproof]{\ref{sectionproof}} below.  The second column of diagram \hyperref[figconiveauprecise2]{4.1.1} supplies rigorous and general definitions of the ``arc objects" discussed by Green and Griffiths, just as the fourth column provides rigorous and general definitions of the corresponding tangent objects.  In special cases, the terms in the right column may be expressed in terms of local cohomology invariants of absolute K\"{a}hler differentials; this occurse, for example, when $A$ is the algebra of dual numbers $k[\ee]/(\ee^2)$.  Recent work of Balmer \cite{BalmerNiveauSS00} suffices to define the second column, but theorem \hyperref[maintheorem]{\ref{maintheorem}} also requires similar results for cyclic homology, along with criteria for exactness. For this, we use the general theory developed by Colliot-Th\'el\`ene, Hoobler and Kahn \cite{CHKBloch-Ogus-Gabber97}, which includes a detailed study of coniveau spectral sequences for generalized cohomology theories with supports. Parts of this formalism were developed earlier by Thomason \cite{Thomason85}.  


\section{Algebraic $K$-theory and Negative Cyclic Homology as Cohomology Theories with Supports.}
\label{sectioncohomsupports}

\subsection{Preliminary Remarks.}
\label{subsectionprelimcohom}

The top row of the coniveau machine appearing in diagram \hyperref[figconiveauprecise2]{4.1.1} above is defined in terms of the sheaves of algebraic $K$-theory $\ms{K}_p(X)$, ``augmented $K$-theory" $\ms{K}_p(X_A)$, relative $K$-theory $\ms{K}_p(X_A,\mm)$, and relative negative cyclic homology $\ms{HN}_p(X_A,\mm)$.   The remaining rows may be constructed in a functorial manner from the top row, but this requires the use of $K$-theory groups and negative cyclic homology groups {\it with supports}. This is because the lower rows depend on information associated with proper subsets of $X$.  From the perspective of Chow groups, this is not very surprising, since Chow groups are intimately related to certain closed subsets of $X$; namely the supports of algebraic cycles.  

In this section, we show that Bass-Thomason (nonconnective) $K$-theory and negative cyclic homology are cohomology theories with supports in the sense of Colliot-Th\'el\`ene, Hoobler, and Kahn \cite{CHKBloch-Ogus-Gabber97}, and therefore supply the groups necessary to build the coniveau machine.  In general, a cohomology theory with supports is a special family of contravariant functors from a category of pairs $(X,Z)$, where $X$ is a scheme and $Z$ is a closed subset,  to an abelian category.  In many important cases, such a theory may be defined in terms of a {\it substratum}, which is a special contravariant functor from a category of pairs to the category of chain complexes of objects over an abelian category, or to an appropriate category of spectra, in the sense of algebraic topology.  The substrata of principal interest for studying the infinitesimal theory of Chow groups are the spectrum-valued functor $\mathbf{K}$ of Bass-Thomason $K$-theory, and the spectrum-valued functor $\mathbf{HN}$ of negative cyclic homology.  The corresponding cohomology theories with supports yield group-valued functors $K_p$ and $HN_p$.   The material in this section comes from SGA \cite{SGA671}, Thomason \cite{Thomason-Trobaugh90}, Colliot-Th\'el\`ene, Hoobler, and Kahn \cite{CHKBloch-Ogus-Gabber97}, Corti\`{n}as, Haesemeyer, Schlichting, and Weibel \cite{WeibelCycliccdh-CohomNegativeK06}, Keller \cite{KellerCycHomofDGAlgebras96}, \cite{KellerCyclicHomologyofExactCat96}, and \cite{KellerCyclicHomologyofSchemes98}, and Schlichting \cite{SchlichtingKTheory11}.  The exposition draws heavily on Dribus \cite{DribusDissertation}.



\subsection{Perfect Complexes.}
\label{subsectionperfect}

Both Bass-Thomason $K$-theory and negative cyclic homology of schemes may be defined in terms of the category of perfect complexes on a quasi-compact, separated scheme.  Here we briefly introduce perfect complexes, which may be viewed as ``generalized complexes of vector bundles."   

\begin{definition}\label{defiperfect} Let $Y$ be a quasi-compact, separated scheme, and let $Z$ be a closed subset of $Y$, such that $Y-Z$ is also quasi-compact.  A complex of quasicoherent $\ms{O}_Y$-modules is called perfect if there exists a covering $\bigcup_{i\in I} U_i$ of $Y$ by open affine subschemes, for some index set $I$, such that the restriction of each complex to $U_i$ is quasi-isomorphic to a bounded complex of vector bundles for each $i$. We denote by $\textsf{Perf}_Z(Y)$ the category of perfect complexes on $Y$ acyclic over $Z$. 
\end{definition}

The category $\textsf{Perf}_Z(Y)$ is a {\it complicial exact category with weak equivalences;} in this case, the weak equivalences are the quasi-isomorphisms.   For justification of this statement, we refer to Schlichting \cite{SchlichtingKTheory11}; see in particular 3.2.9 and the associated discussion.  Due to cardinality issues, Weibel et al. \cite{WeibelCycliccdh-CohomNegativeK06} choose to work with a distinguished subcategory of $\textsf{Perf}_Z(Y)$, chosen to be an {\it exact differential graded category}, and therefore admitting Keller's machinery of {\it localization pairs} \cite{KellerCycHomofDGAlgebras96}, \cite{KellerCyclicHomologyofExactCat96}, \cite{KellerCyclicHomologyofSchemes98}.  Similarly, Thomason \cite{Thomason-Trobaugh90} considers a variety of different subcategories of $\textsf{Perf}_Z(Y)$, all leading to the same $K$-theory spectra.  These issues are well-understood, and are peripheral to our development; we mention them only to alert the fastidious reader that $\textsf{Perf}_Z(Y)$ may really denote a ``suitable" category of perfect complexes in certain cases.


\subsection{Algebraic $K$-Theory.}
\label{subsectionKschemes}

The first column of the coniveau machine in figure \hyperref[figconiveauprecise2]{4.1.1} above is the Cousin resolution of the sheaf $\ms{K}_p(X)$, whose $p$th cohomology group is isomorphic to the Chow group $CH^p(X)$, by the fundamental isomorphism \hyperref[equblochintro]{\ref{equblochintro}} of Bloch, Gersten, and Quillen.   Since the variety $X$ is assumed to be nonsingular, Quillen's $K$-theory is adequate to define this column; in fact, the Cousin resolution of $\ms{K}_p(X)$ may be expressed in terms of the Quillen $K$-groups $K_{p-d}\big(k(x)\big)$ of the residue fields $k(x)$, as described in section \hyperref[subsecGGconiveau]{\ref{subsecGGconiveau}} above.  To properly treat the infinitesimal theory of Chow groups, however, it is necessary to consider singular schemes; in particular, the infinitesimal thickenings $X_A$.  For this purpose, Quillen's $K$-theory is inadequate.  Instead, we use Bass-Thomason $K$-theory \cite{Thomason-Trobaugh90}, which possesses better formal properties in this more general setting.  These properties may be summarized by the statement that the spectrum-valued functor $\mbf{K}$ is an effaceable substratum for a cohomology theory with supports.  We explain the precise meaning of this statement below.  First, for the convenience of the reader, we say a few words about Bass-Thomason $K$-theory. 

It is useful in this context to work in terms of $K$-theory spectra whenever possible, obtaining the desired $K$-theory groups {\it a posteriori} as homotopy groups, in the usual way.   There are several different ways to define these spectra.  Thomason's original definition, drawing on earlier work of Bass, may be found in \cite{Thomason-Trobaugh90}, definition 6.4.  We rely instead on Schlichting \cite{SchlichtingKTheory11}, who provides a useful modern context. 

\begin{definition}\label{defiKtheory} Let $Y$ be a quasi-compact, separated scheme, and let $Z$ be a closed subset of $Y$, such that $Y-Z$ is also quasi-compact.  
\begin{enumerate}
\item The spectrum $\mbf{K}(Y\tn{ on } Z)$ of Bass-Thomason $K$-theory on $Y$ with supports in $Z$ is the nonconnective $K$-theory spectrum of the category $\textsf{Perf}_Z(Y)$, understood as a complicial exact category with quasi-isomorphisms as weak equivalences, as described by Schlichting \cite{SchlichtingKTheory11}, 3.2.26, 3.4.  For a point $y\in Y$, not necessarily closed, the spectrum $\mbf{K}(Y\tn{ on } y)$ of Bass-Thomason $K$-theory on $Y$ supported at $y$ is the direct limit $\displaystyle{\varinjlim_{U\ni y}}\ \mathbf{K}(U \tn{ on } \bar{y}\cap U)$. 
\item The $p$th Bass-Thomason $K$-theory group of $Y$ with supports in $Z$ is the $p$th homotopy group $K_p(Y\tn{ on } Z)$ of the spectrum $\mbf{K}(Y\tn{ on } Z)$.  For a point $y\in Y$, not necessarily closed,  the $p$th Bass-Thomason $K$-theory group of $Y$ supported at $y$ is the $p$th homotopy group $K_p(Y\tn{ on } y)$ of the spectrum $\mbf{K}(Y\tn{ on } y)$. 
\end{enumerate}
\end{definition}

The term {\it nonconnective} signifies that $\mbf{K}(Y\tn{ on } Z)$ generally possesses nontrivial homotopy groups in negative degrees.   The spectrum $\mbf{K}(Y\tn{ on } Y)$ is abbreviated by $\mbf{K}(Y)$; its $p$th homotopy group is the ``ordinary $K$-group" $K_p(Y)$.   We commit a slight abuse by referring to the spectra $\mbf{K}(Y\tn{ on } y)$ and the groups $K_p(Y\tn{ on } y)$ as ``supported at $y$."


\subsection{Negative Cyclic Homology.}
\label{subsectionnegcyclic}

The fourth column of the coniveau machine in diagram \hyperref[figconiveauprecise2]{4.1.1} above is the Cousin resolution of the relative negative cyclic homology sheaf $\ms{HN}_p(X_A,\mm)$.  The $p$th sheaf cohomology group of this sheaf is isomorphic to the group $\widehat{CH}^p (X)(A)$, by our main theorem \hyperref[maintheorem]{\ref{maintheorem}}.  As in the case of algebraic $K$-theory, it is convenient here to work in terms of spectra, and there are again several variations for how to construct these.   The theory ultimately goes back to Weibel \cite{WeibelCyclicHomologySchemes91}, who proved uniqueness results for cyclic homology theories of schemes soon after the discovery of cyclic homology by Connes  et al.  We choose instead to follow Keller's modern approach of {\it localization pairs} \cite{KellerCycHomofDGAlgebras96}, \cite{KellerCyclicHomologyofExactCat96}, \cite{KellerCyclicHomologyofSchemes98}.  This approach produces {\it mixed complexes} $C(Y\tn{ on } Z)$, which may be used to construct negative cyclic homology spectra as described in Corti\`{n}as, Haesemeyer, Schlichting, and Weibel \cite{WeibelCycliccdh-CohomNegativeK06}, section 2.   The same paper gives a good description of the relevant localization pairs.  These are of the form $\big(\textsf{Perf}_Z(Y),\textsf{Ac}\big)$, where $\textsf{Perf}_Z(Y)$ is the category of perfect complexes defined in section \hyperref[subsectionperfect]{\ref{subsectionperfect}} above,  and $\textsf{Ac}$ is its {\it acyclic subcategory.} See \cite{WeibelCycliccdh-CohomNegativeK06}, examples 2.7 and 2.8, for more details.  

\begin{definition}\label{defiKnegcyc} Let $Y$ be a quasi-compact, separated scheme, and let $Z$ be a closed subset of $Y$ such that $Y-Z$ is also quasi-compact.  
\begin{enumerate}
\item The spectrum $\mbf{HN}(Y\tn{ on } Z)$ of negative cyclic homology on $Y$ with supports in $Z$ is the spectrum given by applying the Eilenberg-MacLane functor to the complex
\[\tn{Tot}\big(...\rightarrow0\rightarrow C(Y\tn{ on } Z)\overset{B}{\rightarrow} C(Y\tn{ on } Z)[-1]\overset{B}{\rightarrow} C(Y\tn{ on } Z)[-2]\overset{B}{\rightarrow} ...\big),\]
where $C(Y\tn{ on } Z)$ is Keller's mixed complex for the localization pair $\big(\textsf{Perf}_Z(Y),\textsf{Ac}\big)$.  For a point $y\in Y$, not necessarily closed, the spectrum $\mbf{HN}(Y\tn{ on } y)$ of negative cyclic homology on $Y$ supported at $y$ is the direct limit $\displaystyle{\varinjlim_{U\ni y}}\ \mathbf{HN}(U \tn{ on } \bar{y}\cap U)$. 
\item The $p$th negative cyclic homology group of $Y$ with supports in $Z$ is the $p$th homotopy group $HN_p(Y\tn{ on } Z)$ of the spectrum $\mbf{HN}(Y\tn{ on } Z)$.  For a point $y\in Y$, not necessarily closed,  the $p$th negative cyclic homology group of $Y$ supported at $y$ is the $p$th homotopy group $HN_p(Y\tn{ on } y)$ of the spectrum $\mbf{HN}(Y\tn{ on } y)$. 
\end{enumerate}
\end{definition}

Like the Bass-Thomason $K$-theory functor $\mbf{K}$, the spectrum-valued negative cyclic homology functor $\mbf{HN}$ is an effaceable substratum for a cohomology theory with supports, as shown below.


\subsection{Cohomology Theories with Supports; Substrata.}\label{subsectioncohomsupports}

The necessity of considering cohomology theories with supports, such as algebraic $K$-theory and negative cyclic homology, in the study of Chow groups, arises ultimately from the fact that algebraic cycles define closed subsets of their ambient varieties.  Cohomology theories with supports facilitate the ``sorting" of information contained in global objects, such as sheaves, in terms of these closed subsets.   Substrata are special functors that serve as ``precursors" for cohomology theories with supports. 

We begin this subsection with a few words about the source category for a cohomology theory with supports.  Let $k$ be a field of characteristic zero, and let $\textsf{S}_k$ be a full subcategory of the category of schemes over $k$, stable under \'etale extensions.    Assume that $\textsf{S}_k$ includes the prime spectrum $\tn{Spec }k$ of $k$, and that whenever a scheme $Y$ belongs to  $\textsf{S}_k$, the projective space $\mbb{P}^1(Y):=\mbb{P}^1(\ZZ)\times_{\tn{Spec } (\ZZ)}Y$ also belongs to $\textsf{S}_k$.  Examples of such categories include the category of separated schemes over $k$ and the category of nonsingular schemes over $k$.  Given such a category $\textsf{S}_k$, let $\textsf{P}_k$ be the category of pairs $(Y, Z)$, where $Y$ belongs to $\textsf{S}_k$, and where $Z$ is a closed subset of $Y$.   A morphism of pairs $f: (Y', Z') \to (Y, Z)$ is a morphism $f: Y' \to Y$ in $\textsf{S}_k$ such that $f^{-1}(Z) \subset Z'$.  

The defining property of a cohomology theory with supports over the category of pairs $\textsf{P}_k$ is the existence of a certain long exact sequence of cohomology groups with supports for every {\it triple} $(Y,Z,W)$, where $Y$ belongs to the distinguished category of schemes $\textsf{S}_k$, and $Z$ and $W$ are closed subsets of $Y$ such that $W \subseteq Z \subseteq Y$.  The following definition makes this precise:

\begin{definition}\label{deficohomsupports} Let $k$ be a field of characteristic zero, and let $\textsf{S}_k$ be a category of schemes over $k$ satisfying the conditions given above.  A {\bf cohomology theory with supports} over a category of pairs $\textsf{P}_k$ is a family $h:=\{h^n\}_{n\in\mathbb{Z}}$ of contravariant functors
 \[h^n:(Y, Z) \mapsto h^n (Y\tn{ \footnotesize{on} } Z)\]
 from $\textsf{P}_k$ to an abelian category $\textsf{A}$, satisfying the following condition: for any triple $(Y,Z,W)$, where $W \subseteq Z \subseteq Y$, and where $Z$ and $W$ are closed in $Y$, there exists a long exact sequence:
 \[...\longrightarrow h^n(Y\tn{ \footnotesize{on} } W)\overset{i^n}{\longrightarrow}h^n(Y\tn{ \footnotesize{on} } Z)\overset{j^n}{\longrightarrow}h^n(Y-W\tn{ \footnotesize{on} } Z-W)\overset{d^n}{\longrightarrow}h^{n+1}(Y\tn{ \footnotesize{on}  } W)\longrightarrow...\]
where the maps $i^n$ and $j^n$ are induced by the morphisms of pairs $(Y,Z)\leftarrow(Y,W)$ and $(Y,Z)\leftarrow(Y-W,Z-W)$, and where $d^n$ is the $n$th connecting morphism.
\end{definition}

We may also use the notation $\mathbb{H} ^{n} _{Z}(Y, h) $ for $h^{n}(Y \tn{ on } Z)$, since these groups generalize the classical local cohomology invariants of complexes of sheaves on a scheme.   For a point $y\in Y$, not necessarily closed, we also define the ``punctual invariant" $h^{n}(Y \tn{ on } y)$ to be the limit group $\displaystyle{\varinjlim_{y \in U}}\ h^{n}\big(U \fsz{\tn{ on }} \overline{\{y\}}\cap U\big)$. 

 The basic observation that Bass-Thomason nonconnective $K$-theory and negative cyclic homology are in fact cohomology theories with supports is a consequence of the fact that the corresponding spectrum-valued functors $\mbf{K}$ and $\mbf{HN}$ are substrata, as noted below.  The family $\textsf{Co}(\textsf{P}_k)$ of all cohomology theories with supports on $\textsf{P}_k$, together with their natural transformations, is a contravariant functor category on $\textsf{P}_k$, where natural transformations between cohomology theories with supports are the morphisms in $\textsf{Co}(\textsf{P}_k)$.  An important example of such a morphism is the Chern character between Bass-Thomason nonconnective $K$-theory and negative cyclic homology.  
 
 We now turn to the subject of substrata.  A substratum is a special functor from a suitable category $\textsf{S}_k$ of schemes to a category of complexes or spectra.  Corresponding cohomology functors are given by taking cohomology groups of complexes or homotopy groups of spectra, as explained below. 

\begin{definition}\label{defisubstratum}  Let $k$ be a field of characteristic zero, and let $\textsf{S}_k$ be a category of schemes over $k$ satisfying the conditions given at the beginning of this section.  A {\bf substratum} on $\textsf{S}_k$ is a contravariant functor $C:Y \mapsto C(Y)$ from $\textsf{S}_k$ to the category $\textsf{Ch}(\textsf{A})$ of chain complexes of objects of an abelian category $\textsf{A}$, or to an appropriate category $\textsf{E}$ of spectra.   
\end{definition}

For each pair $(Y,Z)$ belonging to the category of pairs $\textsf{P}_k$ over $\textsf{S}_k$, we may define a ``complex or spectrum with supports" $C(Y\tn{ on } Z)$, by taking the homotopy fiber of the map of complexes or spectra $C(Y) \to C(Y-Z)$.   For a point $y\in Y$, not necessarily closed, we also define the ``punctual invariant" $C(Y \tn{ on } y)$ to be the limit object $\displaystyle{\varinjlim_{y \in U}}\ C\big(U \tn{ on } \overline{\{y\}}\cap U\big)$.

A substratum $C$ serves as a ``precursor" for a cohomology theory with supports, in the following sense: for any triple $W \subseteq Z \subseteq Y$, where $Z$ and $W$ are closed in $Y$, the corresponding complexes or spectra with supports fit together to give short exact sequences
\begin{equation}\label{SEScomplexorspectra}
 0 \longrightarrow C(Y\tn{ on } W) \longrightarrow C(Y\tn{ on  } Z) \longrightarrow C(Y-W\tn{ on } Z-W)\longrightarrow 0.
\end{equation}

A cohomology theory with supports may then be defined by taking cohomology groups of complexes or homotopy groups of spectra: 
\begin{definition}\label{defisubstratumtocohom} Let $C$ be a substratum on an appropriate category of schemes $\textsf{S}_k$, as in definition \hyperref[defisubstratum]{\ref{defisubstratum}}.  If the target category of $C$ is a category of complexes $\textsf{Ch}(\textsf{A})$, define functors $h^n$ from $\textsf{P}_k$ to $\textsf{A}$ by taking cohomology of complexes:
\[h^n(Y\tn{ on } Z) := H^n\big(C(Y\tn{ on } Z)\big).\]
If the target category of $C$ is a category of spectra $\textsf{E}$, define functors functors $h^n$ from $\textsf{P}_k$ to $\textsf{A}$ by taking homotopy groups of spectra:
\[h^n(Y\tn{ on } Z) := \pi_{-q}\big(C(X\tn{ on } Z)\big).\]
\end{definition}

By basic homological algebra, the short exact sequence of complexes or spectra in equation \hyperref[SEScomplexorspectra]{\ref{SEScomplexorspectra}} induces a long exact sequence of groups $h^n(Y\tn{ on } Z)$, as in definition \hyperref[deficohomsupports]{\ref{deficohomsupports}}.  Hence, the functors $h^n$ indeed define a cohomology theory with supports.   It is an easy exercise to show that taking cohomology groups of complexes or homotopy groups of spectra coincides with taking direct limits in this context, so that the punctual invariants  $h^n(Y \tn{ on } y)$ coincide with the cohomology groups or homotopy groups of the punctual invariants  $C(Y \tn{ on } y)$.  Thomason \cite{Thomason85} introduced local cohomology invariants $\mathbb{H} ^{n} _{Z}(Y, C)  := \pi _{-n} \big(\mathbb{H}  _{Z}(Y, C) \big)$, where $C$ is a presheaf of spectra on $X$.  Since the historical connection to classically defined local cohomology is important to our viewpoint, this notation is sometimes useful as an alternative to $h^n(Y\tn{ on } Z)$.   

For the sake of completeness, we include the basic observation that the spectrum-valued functors $\mbf{K}$ and $\mbf{HN}$ are in fact substrata.

\begin{lemma}\label{lemKHNsubstrat} The functors $\mbf{K}$ and $\mbf{HN}$ are substrata.  Hence, Bass-Thomason nonconnective $K$-theory and negative cyclic homology are cohomology theories with supports. 
\end{lemma}
\begin{proof} The only requirement is the basic condition that  $\mbf{K}$ and $\mbf{HN}$ be contravariant complex-valued or spectrum-valued functors.  
\end{proof}

The family $\textsf{Sub}_{\textsf{S}_k}$ of all substrata on $\textsf{S}_k$, together with their natural transformations, is a contravariant functor category on $\textsf{S}_k$.   The Chern character between Bass-Thomason $K$-theory and negative cyclic homology may be understood at the level of substrata, by working with the spectra $\mbf{K}$ and $\mbf{HN}$. 


\subsection{Augmented and Relative Variants.}\label{subsectionaugmented}

Colliot-Th\'el\`ene, Hoobler, and Kahn \cite{CHKBloch-Ogus-Gabber97} devote considerable attention to the procedure of modifying a given cohomology theory with supports $h$ to define a new such theory $h'$ satisfying specified conditions.  An analogous procedure may be applied to substrata.  The rationale for this is that useful properties of the ``old" theory $h$ may sometimes be easily transferred to the ``new" theory $h'$.  An important example of such a property is {\it effaceability,} discussed in more detail below.  In this paper, $h'$ will usually be an ``augmented" or ``relative" version of $h$.  Of particular interest is the case where $h'$ is defined by ``multiplying by a fixed separated scheme."  

\begin{definition}\label{defnewoutofoldaug} Let $k$ be a field of characteristic zero, and let $\textsf{S}_k$ be a category of schemes over $k$ satisfying the conditions given at the beginning of section \hyperref[subsectioncohomsupports]{\ref{subsectioncohomsupports}}.   Let $S$ be a separated scheme over $k$.  
\begin{enumerate} 
\item Let $h$ be a cohomology theory with supports on the category $\textsf{P}_k$ of pairs over $\textsf{S}_k$, with values in an abelian category $\textsf{A}$.   Define a family of functors $h^{S}=\{h^{S,n}\}_{n\in\ZZ}$ on the category of pairs over $\textsf{S}_k$ as follows:
\begin{equation}\label{equnewoutofoldcohom}h^{S,n}(Y\tn{ on } Z):=h^n(Y\times_k S\tn{ on } Z\times_k S ).\end{equation}
where the notation $\times_k$ for the fiber product means $\times_{\tn{Spec }(k)}$.  We call the family $h^{S}$ the {\bf augmented version of $h$ with respect to $S$.} 
\item  Let $C$ be a substratum on $\textsf{S}_k$, with values in an abelian category $\textsf{A}$.   Define a functor $C^{S}$ on $\textsf{S}_k$ as follows:
\begin{equation}\label{equnewoutofoldsubstrat}C^S(Y):=C(Y\times_k S),\end{equation}
where the notation $\times_k$ for the fiber product means $\times_{\tn{Spec }(k)}$.  We call the functor $C^{S}$ the {\bf augmented version of $C$ with respect to $S$.} 
\end{enumerate}
\end{definition}

It is an easy exercise to show that $h^S$ and $C^S$ satisfy the properties of a cohomology theory with supports and a substratum, respectively.  See  \cite{CHKBloch-Ogus-Gabber97}, section 5.5, for details.  

Given an ``absolute theory," such as $h$ or $C$, and an ``augmented theory," such as $h^S$ or $C^S$, one may define corresponding ``relative theory" by taking homotopy fibers.   However, it is logically preferable to approach relative constructions without prior reference to an augmented theory.   Hence, we use the following definition, confining our attention to the case of substrata:  

\begin{definition}\label{defnewoutofoldrel} Let $k$ be a field of characteristic zero, and let $\textsf{S}_k$ be a category of schemes over $k$ satisfying the conditions given at the beginning of section \hyperref[subsectioncohomsupports]{\ref{subsectioncohomsupports}}.   Let $Y$ be a scheme belonging to $\textsf{S}_k$, and let $\ms{I} \subset \ms{O}_Y$ be a functorially defined sheaf of ideals.  Let $X$ be the corresponding quotient scheme.  Let $C$ be a substratum on $\textsf{S}_k$.   Define $C^{\ms{I}}(Y)$ to be the homotopy fiber of the map
\begin{equation}\label{equnewoutofoldsubstrat}C(Y)\longrightarrow C(X).\end{equation}
We call the functor $C^{\ms{I}}$ the {\bf relative version of $C$ with respect to $\ms{I}$.} 
\end{definition}

When referring to the individual complexes or spectra $C^{\ms{I}}(Y)$, instead of the functor $C^{\ms{I}}$ itself, it is convenient for historical reasons to use the notation $C(Y,\ms{I})$ for $C^{\ms{I}}(Y)$.   Here, $Y$ is usually a thickening $X\times_k S$ of a variety $X$, where $S$ is the prime spectrum of an Artinian local $k$-algebra $A$ with maximal ideal $\mm$ and residue field $k$, and where we take $\ms{I}=\ms{O}_X\otimes _k\mm$.  In this case,  we use the shorthand $C(Y,\mm)$ for $C^{\ms{I}}(Y)$.  Additional variations may be obtained  by combining theories with supports and relative theories in the obvious way.  The prototypical example is the relative $K$-theoretic spectrum for $Y=X\times_k S$ with supports in a closed subset $Z$, relative to an ideal $\mm$ in the sense above, which we denote by $\mbf{K}( Y\tn{ on } Z,\mm)$.  

We now take a brief detour to examine more closely the local structure underlying the relative theories of principal interest in this paper.  In particular, we consider a ring $R$ and an ideal $I$ of $R$, which are understood to locally represent the scheme $Y$ and sheaf of ideals $\ms{I}$ of definition \hyperref[defnewoutofoldrel]{\ref{defnewoutofoldrel}}. The quotient ring $R/I=S$ locally represents the quotient scheme $X$. Since we are interested in infinitesimal structure, we focus on the case where $I$ is a nilpotent ideal.   In this context, $R$ is called a split nilpotent extension of $S$, and the pair $(R,I)$ is called a split nilpotent pair.  The functorial behavior of such pairs is important in what follows, so we need the following definition:

\begin{definition}\label{defisplitnilpotentpairs} The {\bf category of split nilpotent pairs} $\tn{\textsf{Nil}}$ is the category whose objects are pairs $(R,I)$, where $R$ is a ring and $I$ is an ideal of $R$, and whose morphisms $(R,I)\rightarrow(R',I')$ are ring homomorphisms $R\rightarrow R'$ such that $I$ maps into $I'$.  
\end{definition}

We remark that the objects $(R,I)$ of $\textsf{Nil}$ are {\it not} local versions of the objects of the category of pairs $\mbf{P}_k$ over a distinguished category of schemes $\mbf{S}_k$.   In this context, a ``local version of an object of $\mbf{P}_k$" would be a pair $(S,J)$, where $S$ is a ring locally representing a suitable scheme $X$, often a variety, and $S$ is an ideal locally representing a closed subset $Z$ of $X$.  Hence, when dealing with supports and relative constructions simultaneously at a local level, we are really working with {\it triples} consisting of a ring and two ideals. 

The principal reason for introducing the category of split nilpotent pairs $\textsf{Nil}$ is to enable a suitable treatment of the relative algebraic Chern character, introduced in section \hyperref[cherncharacter]{\ref{cherncharacter}} below.  In particular, relative algebraic $K$-theory and relative negative cyclic homology may be viewed as functors from $\textsf{Nil}$ to the category of abelian groups, and the relative Chern character is an isomorphism of functors between these two theories. 


\section{Effaceability}\label{effacement}

\subsection{Preliminary Remarks.}\label{subsectioneffacementprelim}

Effaceability is a technical condition involving the behavior of a cohomology theory with supports, or a substratum, with respect to open neighborhoods of finite collections of points in smooth affine schemes in a distinguished category $\textsf{S}_k$ of the type introduced at the beginning of section \hyperref[subsectioncohomsupports]{\ref{subsectioncohomsupports}}.  Effaceability plays a vital role in the construction of the coniveau machine by guaranteeing the exactness of certain sheafified Cousin complexes, which serve as flasque resolutions of $K$-theory and cyclic homology sheaves, as described below.   In particular, this provides a method of computing ``generalized deformation groups" and formal completions of Chow groups. 

\begin{definition}\label{defieffaceability} Let $k$ be a field of characteristic zero, and let $\textsf{S}_k$ be a category of schemes over $k$ satisfying the conditions given at the beginning of section \hyperref[subsectioncohomsupports]{\ref{subsectioncohomsupports}}.  Let $X\in\textsf{S}_k$ be an affine scheme, and let $\{t_1,...,t_r\}$ be a finite set of points of $X$.  
\begin{enumerate}
\item A cohomology theory with supports $h=\{h^n\}_{n\in\mathbb{Z}}$ on the category of pairs $\textsf{P}_k$ over $\textsf{S}_k$ is called {\bf effaceable}\footnotemark\footnotetext{Colliot-Th\'el\`ene, Hoobler, and Kahn \cite{CHKBloch-Ogus-Gabber97} call this condition {\bf strict effaceability} (definition 5.1.8, page 28) but I drop the adjective ``strict" since this is the only such property used here. } {\bf at} $\{t_1,...,t_r\}$ if, given any integer $p\ge0$, any open neighborhood $W$ of $\{t_1,...,t_r\}$ in $X$, and any closed subset $Z\subset W$ of codimension at least $p+1$, there exists a smaller open neighborhood $U\subseteq W$ of $\{t_1,...,t_r\}$ in $X$ and a closed subset $Z'\subseteq W$, containing $Z$, with $\tn{codim}_{W}(Z')\ge p$, such that the map 
\[h^n(U\tn{ \footnotesize{on} } Z\cap U)\rightarrow h^n(U\tn{ \footnotesize{on} } Z'\cap U)\]
is zero for all $n\in\mathbb{Z}$. The cohomology theory with supports $h$ is called {\bf effaceable} if this condition is satisfied for any nonsingular $X$ in $\textsf{S}_k$, and any choice of $\{t_1,...,t_r\}$.  
\item A substratum $C$ on $\textsf{S}_k$ is called {\bf effaceable at} $\{t_1,...,t_r\}$ if, given $p, W,$ and $Z$ as above, there exists $U$ and $Z'$ as above such that the map of substrata
\[C(U\tn{ \footnotesize{on} } Z\cap U)\rightarrow C(U\tn{ \footnotesize{on} } Z'\cap U)\]
is nullhomotopic.  The substratum $C$ is called {\bf effaceable} if this condition is satisfied for any nonsingular $X$ in $\textsf{S}_k$, and any choice of $\{t_1,...,t_r\}$.  
\end{enumerate}
\end{definition}

Colliot-Th\'el\`ene, Hoobler, and Kahn \cite{CHKBloch-Ogus-Gabber97} prove several different versions of an {\it effacement theorem} giving conditions under which cohomology theories with supports and substrata are effaceable.   These different versions involve a variety of different criterion, offering a spectrum of choices affording different balances of generality and ease of application.  Algebraic $K$-theory and negative cyclic homology satisfy somewhat stronger hypotheses than those required for the most general effacement theorems, so relatively ``easy" criterion may be used for these particular theories.   The criterion we use for cohomology theories with supports are called {\it \'etale excision} and the {\it cohomological projective bundle condition}.  The conditions we use for substrata are called the {\it \'etale Mayer-Veitoris condition} and the {\it projective bundle condition for substrata}.


\subsection{Effaceability Criteria for Cohomology Theories with Supports.}\label{effaceabilitycriterioncohom}

Let $h$ and $h'$ be cohomology theories with supports on the category of pairs $\textsf{P}_k$ over a distinguished category of schemes $\textsf{S}_k$ over a field $k$ of characteristic zero, satisfying the conditions given at the beginning of section \hyperref[subsectioncohomsupports]{\ref{subsectioncohomsupports}}.  Following Colliot-Th\'el\`ene, Hoobler, and Kahn \cite{CHKBloch-Ogus-Gabber97}, we specify effaceability criterion  for $h$, called \'etale excision and the projective bundle condition.  

\begin{definition}\label{defiCOH1} A cohomology theory with supports $h$ satisfies {\bf \'etale excision} if $h$ is additive, and if for any diagram of the form shown below, where $f$ is \'etale and $f^{-1}(Z)\rightarrow Z$ is an isomorphism, the induced map $f^*:h^q(Y\tn{\footnotesize{ on }} Z) \rightarrow h^q(Y'\tn{\footnotesize{ on }} Z)$ is an isomorphism for all $q$:

\begin{pgfpicture}{0cm}{0cm}{17cm}{1.2cm}
\begin{pgftranslate}{\pgfpoint{5.5cm}{-1.2cm}}

\pgfputat{\pgfxy(2.75,2.5)}{\pgfbox[center,center]{$X'$}}
\pgfputat{\pgfxy(1.5,1)}{\pgfbox[center,center]{$Z$}}
\pgfputat{\pgfxy(2.75,1)}{\pgfbox[center,center]{$X$}}
\pgfputat{\pgfxy(2.95,1.8)}{\pgfbox[center,center]{$f$}}
\pgfsetendarrow{\pgfarrowlargepointed{3pt}}
\pgfxyline(2.75,2.2)(2.75,1.3)
\pgfxyline(1.8,1)(2.5,1)
\pgfxyline(1.6,1.3)(2.4,2.2)
\end{pgftranslate}

\end{pgfpicture}

\end{definition}

The closely-related condition of {\it Zariski excision} is defined by letting $f$ run over all Zariski open immersions in the statement of definition \hyperref[defiCOH1]{\ref{defiCOH1}}.  This condition is sometimes useful as well; for example, it permits a stronger version of theorem \hyperref[theoremconiveaudescent]{\ref{theoremconiveaudescent}} below. 

Some preliminary explanation, following Colliot-Th\'el\`ene, Hoobler, and Kahn \cite{CHKBloch-Ogus-Gabber97}, section 5.4, will be helpful before stating the projective bundle condition.  Let $V$ be an open subset of the $n$-dimensional affine space $\AA_k^n$ over $k$ for some $n$, and let the diagram

\begin{pgfpicture}{0cm}{0cm}{17cm}{1.9cm}
\begin{pgftranslate}{\pgfpoint{5cm}{-1cm}}
\pgfputat{\pgfxy(1.5,2.5)}{\pgfbox[center,center]{$\AA_V^1$}}
\pgfputat{\pgfxy(3.25,2.5)}{\pgfbox[center,center]{$\PP_V^1$}}
\pgfputat{\pgfxy(5,2.5)}{\pgfbox[center,center]{$V$}}
\pgfputat{\pgfxy(3.25,1)}{\pgfbox[center,center]{$V$}}
\pgfputat{\pgfxy(2.3,2.75)}{\pgfbox[center,center]{$j$}}
\pgfputat{\pgfxy(4.25,2.7)}{\pgfbox[center,center]{$s_\infty$}}
\pgfputat{\pgfxy(2.1,1.65)}{\pgfbox[center,center]{$\pi$}}
\pgfputat{\pgfxy(3.5,1.9)}{\pgfbox[center,center]{$\tilde{\pi}$}}
\pgfputat{\pgfxy(4.45,1.65)}{\pgfbox[center,center]{$=$}}
\pgfsetendarrow{\pgfarrowlargepointed{3pt}}
\pgfxyline(1.85,2.5)(2.85,2.5)
\pgfxyline(4.75,2.5)(3.65,2.5)
\pgfxyline(1.7,2.2)(3,1.2)
\pgfxyline(4.7,2.2)(3.5,1.2)
\pgfxyline(3.25,2.2)(3.25,1.3)
\end{pgftranslate}
\end{pgfpicture}

represent the inclusion of $\AA_k^1$ and the section at infinity into the projective space $\PP_V^1$ over $V$.  Let $F$ be a closed subset of $V$. Let Assume that $h$ and $h'$ share a common target category $\textsf{A}$, and that for any pair $(X,Z)$ there exists a map
\begin{equation}\label{equprojbundlecohom1}\xymatrixcolsep{2pc}\xymatrix{\tn{Pic}(X)\ar[r]&\tn{Hom}_{\mbf{A}}\big(h'(X\tn{\footnotesize{ on }} Z),h(X\tn{\footnotesize{ on }} Z)\big),}\end{equation}
which is functorial for pairs $(X,Z)$.  Taking $X=\PP_V^1$ and $Z=\PP_F^1$, there is a homomorphism
\begin{equation}\label{equprojbundlecohom2}\xymatrixcolsep{4.5pc}\xymatrix{h'(\PP_V^1\tn{\footnotesize{ on }} \PP_F^1) \ar[r]&h(\PP_V^1\tn{\footnotesize{ on }} \PP_F^1)}.\end{equation}
Composing with $\tilde{\pi}^*$, there is a homomorphism 
\begin{equation}\label{equprojbundlecohom3}\xymatrixcolsep{4pc}\xymatrix{h'(V\tn{\footnotesize{ on }} F)  \ar[r]^{\alpha_{V,F}} &h(\PP_V^1\tn{\footnotesize{ on }} \PP_F^1)},\end{equation}
which is functorial for pairs $(V,F)$.

\begin{definition}\label{defiCOH5} The pair of cohomology theories with supports $(h,h')$ satisfies the (cohomological) {\bf projective bundle formula}, if for $V$, $F$, $\tilde{\pi}$ as given above, the natural map
\begin{equation}\label{equprojbundlecohom4}\xymatrixcolsep{4.5pc}\xymatrix{h^q(V\tn{\footnotesize{ on }} F)\oplus h^{'q}(V\tn{\footnotesize{ on }} F) \ar[r]^-{\tilde{\pi}^*,\hspace*{.05cm}\alpha_{V,F}} &h^q(\PP_V^1\tn{\footnotesize{ on }} \PP_F^1)}\end{equation}
is an isomorphism for all $q$.  In particular, if the pair $(h,h)$, satisfies the projective bundle formula, then one says that $h$ satisfies the projective bundle formula. 
\end{definition}


\subsection{Effaceability Criteria for Substrata.}\label{effaceabilitycriterioncohom}

Let $C$ and $C'$ be substrata on $\textsf{S}_k$.  Again following \cite{CHKBloch-Ogus-Gabber97}, we specify a condition involving the behavior of $C$ with respect to \'etale covers of $X$, and a condition involving the behavior of $C$ and $C'$ with respect to projective bundles over subsets of $X$.   

\begin{definition}\label{defiSUB1} The substratum $C$ satisfies the {\bf \'etale Mayer-Vietoris condition} if $C$ is additive, and if and for any diagram of the form shown below on the left, where $f$ is \'etale and $f^{-1}(Z)\rightarrow Z$ is an isomorphism, the commutative square shown below on the right is homotopy cartesian:

\begin{pgfpicture}{0cm}{0cm}{17cm}{2.25cm}
\begin{pgftranslate}{\pgfpoint{2.5cm}{-.75cm}}
\pgfputat{\pgfxy(2.75,2.5)}{\pgfbox[center,center]{$X'$}}
\pgfputat{\pgfxy(1.5,1)}{\pgfbox[center,center]{$Z$}}
\pgfputat{\pgfxy(2.75,1)}{\pgfbox[center,center]{$X$}}
\pgfputat{\pgfxy(2.95,1.8)}{\pgfbox[center,center]{$f$}}
\pgfsetendarrow{\pgfarrowlargepointed{3pt}}
\pgfxyline(2.75,2.2)(2.75,1.3)
\pgfxyline(1.8,1)(2.5,1)
\pgfxyline(1.6,1.3)(2.4,2.2)
\end{pgftranslate}
\begin{pgftranslate}{\pgfpoint{6.5cm}{-.75cm}}
\pgfputat{\pgfxy(1.5,2.5)}{\pgfbox[center,center]{$C(X')$}}
\pgfputat{\pgfxy(4,2.5)}{\pgfbox[center,center]{$C(X'-Z)$}}
\pgfputat{\pgfxy(1.5,1)}{\pgfbox[center,center]{$C(X)$}}
\pgfputat{\pgfxy(4,1)}{\pgfbox[center,center]{$C(X-Z)$}}
\pgfsetendarrow{\pgfarrowlargepointed{3pt}}
\pgfxyline(1.5,1.3)(1.5,2.2)
\pgfxyline(4,1.3)(4,2.2)
\pgfxyline(2.1,2.5)(2.95,2.5)
\pgfxyline(2.1,1)(2.95,1)
\end{pgftranslate}
\end{pgfpicture}

\end{definition}
A useful fact cited in \cite{CHKBloch-Ogus-Gabber97}, lemma 5.1.2, page 25, is that the above square is homotopy cartesian if and only if the induced map
\[\xymatrix{C(X\tn{ \footnotesize{on} } Z) \ar[r]^{f} &C(X'\tn{ \footnotesize{on} } Z)}\]
is a homotopy equivalence. 

Some preliminary explanation, again following Colliot-Th\'el\`ene, Hoobler, and Kahn \cite{CHKBloch-Ogus-Gabber97}, section 5.4, will be helpful before stating the projective bundle condition for substrata.   As in the case of cohomology theories with supports, let $V$ be an open subset of $\AA_k^n$ for some $n$, and let the diagram

\begin{pgfpicture}{0cm}{0cm}{17cm}{2.25cm}
\begin{pgftranslate}{\pgfpoint{5cm}{-.75cm}}
\pgfputat{\pgfxy(1.5,2.5)}{\pgfbox[center,center]{$\AA_V^1$}}
\pgfputat{\pgfxy(3.25,2.5)}{\pgfbox[center,center]{$\PP_V^1$}}
\pgfputat{\pgfxy(5,2.5)}{\pgfbox[center,center]{$V$}}
\pgfputat{\pgfxy(3.25,1)}{\pgfbox[center,center]{$V$}}
\pgfputat{\pgfxy(2.3,2.75)}{\pgfbox[center,center]{$j$}}
\pgfputat{\pgfxy(4.25,2.7)}{\pgfbox[center,center]{$s_\infty$}}
\pgfputat{\pgfxy(2.1,1.65)}{\pgfbox[center,center]{$\pi$}}
\pgfputat{\pgfxy(3.5,1.9)}{\pgfbox[center,center]{$\tilde{\pi}$}}
\pgfputat{\pgfxy(4.45,1.65)}{\pgfbox[center,center]{$=$}}
\pgfsetendarrow{\pgfarrowlargepointed{3pt}}
\pgfxyline(1.85,2.5)(2.85,2.5)
\pgfxyline(4.75,2.5)(3.65,2.5)
\pgfxyline(1.7,2.2)(3,1.2)
\pgfxyline(4.7,2.2)(3.5,1.2)
\pgfxyline(3.25,2.2)(3.25,1.3)
\end{pgftranslate}
\end{pgfpicture}

represent the inclusion of $\AA_k^1$ and the section at infinity into $\PP_V^1$.  Let $F$ be a closed subset of $V$.   Assume $C$ and $C'$ share a common target category $\textsf{E}$ of complexes or spectra, such that for any $X$ in $\textsf{S}_k$ there exists a map
\begin{equation}\label{equprojbundlesub1}\xymatrixcolsep{2pc}\xymatrix{\tn{Pic}(X)\ar[r]&\tn{Hom}_{\mbf{E}}\big(C(X'),C(X)\big),}\end{equation}
which is functorial for $X$.  Taking $X=\PP_V^1$, there exists a map 
\begin{equation}\label{equprojbundlesub2}\xymatrixcolsep{5pc}\xymatrix{ C'(\PP_V^1) \ar[r]&C(\PP_V^1)}.\end{equation}
Hence, composing with $\tilde{\pi}^*$, there exists a map 
\begin{equation}\label{equprojbundlesub3}\xymatrixcolsep{4pc}\xymatrix{ C'(V) \ar[r]^{\alpha_{V}} &C(\PP_V^1)},\end{equation}
functorial for $V$.  For spectra, these maps are in the {\it stable homotopy category.} 

\begin{definition}\label{defiSUB5} The pair of substrata $(C,C')$ satisfies the {\bf projective bundle formula} (for substrata), if for $V$, $\tilde{\pi}$ as given above, the natural maps
\begin{equation}\label{equprojbundlesub3}\xymatrixcolsep{4.5pc}\xymatrix{C(V)\oplus C'(V) \ar[r]^-{\tilde{\pi}^*,\alpha_{V}} &C(\PP_V^1)\hspace*{1cm}\tn{{\it (for complexes)}}}\end{equation}
\begin{equation}\label{equprojbundlesub4}\xymatrixcolsep{4.5pc}\xymatrix{C(V)\vee C'(V) \ar[r]^-{\tilde{\pi}^*,\alpha_{V}} &C(\PP_V^1)\hspace*{1cm}\tn{{\it(for spectra)}}}\end{equation}
are homotopy equivalences.  In particular, if the pair $(C,C)$, satisfies the projective bundle formula, one says that $C$ satisfies the projective bundle formula. 
\end{definition}

\subsection{Effaceability for $K$-Theory and Negative Cyclic Homology}\label{subsectioneffaceabilitynonconnectiveK}

The effaceability conditions for Bass-Thomason $K$-theory follow directly from Thomason's seminal paper \cite{Thomason-Trobaugh90}, particularly {\it Thomason's localization theorem} (\cite{Thomason-Trobaugh90}, theorem 7.4).  The corresponding results for negative cyclic homology are easier, and follow from parenthetical results in Weibel et al.  \cite{WeibelCycliccdh-CohomNegativeK06}. 


Let $\mathbf{K}$ be the functor assigning to a scheme $X$ the Bass-Thomason nonconnective $K$-theory spectrum $\mbf{K}(X)$. 

The following theorem is part of Thomason's localization theorem:

\begin{theorem}\label{thomasonhomotopysequence} Let $X$ be a quasi-compact and quasi-separated scheme.  Let $Z$ be a closed subspace of $X$ such that $X-Z$ is quasi-compact.  Then there is a  homotopy fiber sequence
\begin{equation}\label{equthomasonlocal1}\mathbf{K}(X\tn{ on } Z)\rightarrow \mathbf{K}(X)\rightarrow \mathbf{K}(X-Z).\end{equation}
\end{theorem}
\begin{proof} See \cite{Thomason-Trobaugh90}, Theorem 7.4.
\end{proof}

The following theorem is a preliminary result in the buildup to Thomason's localization theorem:

\begin{theorem}\label{thomasonhomotopyequ} Let $(X,Z)$ be as in the statement of theorem \hyperref[thomasonhomotopysequence]{\ref{thomasonhomotopysequence}}.  Let $f:X'\rightarrow X$ be a map of quasi-compact and quasi-separated schemes which is \'etale, and which induces an isomorphism $f^{-1}(Z)\rightarrow Z$.  Let $Z$ be a closed subspace of $X$ such that $X-Z$ is quasi-compact.  Then the map of spectra 
\begin{equation}\label{equthomasonetal2}f^*:\mathbf{K}(X\tn{ on } Z)\rightarrow \mathbf{K}(X'\tn{ on } Z)\end{equation}
is a homotopy equivalence. 
\end{theorem}
\begin{proof} See Thomason \cite{Thomason-Trobaugh90}, theorem 7.1. 
\end{proof}

The following theorem is a special case of another result leading up to Thomason's localization theorem: 

\begin{theorem}\label{thomasonprojbundle} Let $V$ be an open subset of $\AA_k^n$ for some $n$.  Let $\tilde{V}$ be the trivial bundle of rank $2$ over $V$, and let $\PP_V^1:=\PP(\tilde{V})$ be the corresponding projective space bundle.  Then there is a homotopy equivalence
\[\mathbf{K}(V)\oplus\mathbf{K}(V)\rightarrow \mathbf{K}(\PP_V^1).\]
\end{theorem}
\begin{proof} See Thomason \cite{Thomason-Trobaugh90} theorem 7.3.
\end{proof}

Together, these results suffice to prove the effaceability of Bass-Thomason $K$-theory:

\begin{corollary} $\mbf{K}$ is an effaceable substratum.
\end{corollary}
\begin{proof}  Theorems \hyperref[thomasonhomotopysequence]{\ref{thomasonhomotopysequence}} and \hyperref[thomasonhomotopyequ]{\ref{thomasonhomotopyequ}} together imply that $\mbf{K}$ satisfies the \'etale Mayer-Vietoris condition, and theorem \hyperref[thomasonprojbundle]{\ref{thomasonprojbundle}} implies that $\mbf{K}$ satisfies the projective bundle condition for substrata.
\end{proof}


\label{subsectioneffaceabilitynegativecyclic}

We turn now to negative cyclic homology.  The \'etale Mayer-Vietoris condition is called {\it Nisnevich descent} by Weibel et al. \cite{WeibelCycliccdh-CohomNegativeK06}, page 3.  In more detail, following the terminology of \cite{WeibelCycliccdh-CohomNegativeK06}, an {\it elementary Nisnevich square} is a cartesian square of schemes 
\[\xymatrix{X'\ar[d]_f  &Y'\ar[l]\ar[d] \\
X&Y\ar[l]_i}\]
for which $X\leftarrow Y:i$ is an open embedding, $f:X'\rightarrow X$ is \'etale, and $(X'-Y')\rightarrow(X-Y)$ is an isomorphism.   
A functor $C$ from the category $\textsf{Fin}_k$ of schemes essentially of finite type over a field $k$ with values in a suitable category of spectra satisfies {\it Nisnevich descent,} as defined in \cite{WeibelCycliccdh-CohomNegativeK06}, if, for any such cartesian square, the square of spectra given by applying $C$ is homotopy cartesian.  Weibel et al. \cite{WeibelCycliccdh-CohomNegativeK06} remark that the term ``\'etale descent" used in Weibel and Geller is equivalent to Nisnevich descent for presheaves of $\mathbb{Q}$-modules.

 Letting $Z$ be a closed subset of $X$ and $Y=X-Z$ its complement, there is an open embedding $X\leftarrow Y:i$, which may be completed to a cartesian square.   If $C$ is contravariant, then the Nisnevich descent condition is precisely the \'etale Mayer-Vietoris condition given by Colliot-Th\'el\`ene, Hoobler and Kahn \cite{CHKBloch-Ogus-Gabber97}.  

\begin{theorem} $\mbf{HN}$ is an effaceable substratum.
\end{theorem}
\begin{proof}  For the \'etale Mayer-Vietoris condition, see \cite{WeibelCycliccdh-CohomNegativeK06} theorem 2.9, page 9.  For the projective bundle condition for substrata, see  \cite{WeibelCycliccdh-CohomNegativeK06}, Remark 2.11, page 9.
\end{proof}


\subsection{Effaceability of Augmented and Relative Theories.}\label{effaceabilityaugmented}

The proof of our main theorem requires that the property of effaceability should be stable under passage to augmented and relative variants of cohomology theories with supports. This is verified by the following lemma, adapted from Colliot-Th\'el\`ene, Hoobler, and Kahn \cite{CHKBloch-Ogus-Gabber97}:

\begin{lemma}\label{lemnewoutofold} Let $\textsf{S}_k$ a category of schemes over a field $k$ satisfying the conditions given at the beginning of section \hyperref[subsectioncohomsupports]{\ref{subsectioncohomsupports}}.   Let $S$ be a separated scheme over $k$.  
\begin{enumerate}
\item  Let $h$ be a cohomology theory with supports on the category of pairs $\textsf{P}_k$ over $\textsf{S}_k$, and let $h^S$ be the corresponding augmented theory.  Then if $h$ satisfies \'etale excision and the cohomological projective bundle formula, so does $h^S$.  
\item Let $C$ be a substratum on $\textsf{S}_k$, and let $C^S$ be the corresponding augmented substratum.  Then if $C$ satisfies the \'etale Mayer-Vietoris property and the projective bundle formula for substrata, so does $C^S$.  
\end{enumerate}
\end{lemma}
\begin{proof} See Colliot-Th\'el\`ene, Hoobler, and Kahn \cite{CHKBloch-Ogus-Gabber97} 5.5. 
\end{proof}


\section{Proof of the Main Theorem.}
\label{sectionproof}

\subsection{Coniveau Spectral Sequence.}  In this section we gather together and synthesize the results necessary to prove our main theorem \hyperref[maintheorem]{\ref{maintheorem}}, beginning with the existence and properties of the coniveau spectral sequence.  Let $X$ be a scheme belonging to the distinguished category $\textsf{S}_k$, and let $h$ be a cohomology theory with supports $h$ on a category of pairs $\textsf{P}_k$ over $\textsf{S}_k$.  The procedure of coniveau filtration; i.e., filtration of $X$ by the codimensions of its points, leads to an exact couple involving the cohomology objects $h^n(X \tn{ on }x)$, and thence to the coniveau spectral sequence.  Actually, this construction goes through under much more general conditions than we need for theorem \hyperref[maintheorem]{\ref{maintheorem}}.

\begin{theorem}\label{theoremconiveaudescent} Let $k$ be a field of characteristic zero, and let $\textsf{S}_k$ be a category of schemes over $k$ satisfying the conditions given at the beginning of section \hyperref[subsectioncohomsupports]{\ref{subsectioncohomsupports}}.   Let $C$ be a substratum on $\textsf{S}_k$, and let $h$ be the corresponding cohomology theory with supports on the category of pairs $\textsf{P}_k$ over $\textsf{S}_k$.  Assume that $C$ satisfies \'etale excision and the condition that  $h^n(\emptyset) = 0$ for all $n$.  Let $X\in\textsf{S}_k$ be a noetherian scheme of finite dimension, and let $\ms{H}^n(X)$ be the sheaf on $X$ associated to the presheaf $U\mapsto h^n(U)$.  Then there exists a strongly convergent spectral sequence, called the {\bf coniveau spectral sequence}:
\begin{equation}\label{equconiveauSS}
 E_{1}^{p,q} = \coprod_{x \in X^{p}}h^{p+q}(X \tn{ on }x) \Rightarrow h^{p+q}(X).
\end{equation}
\end{theorem}
\begin{proof}  The construction of the coniveau spectral sequence and proof of the convergence condition in the special case of \'etale cohomology, following Bloch-Ogus \cite{BlochOgus75} and Gabber \cite{GabberPfBlochOgus}, appears in See Colliot-Th\'el\`ene, Hoobler and Kahn \cite{CHKBloch-Ogus-Gabber97}, section  1.  In \cite{CHKBloch-Ogus-Gabber97}, section 5, remark 5.1.3(3), it is observed that the same proof goes through under the stated conditions for any cohomology theory with supports.  Dribus \cite{DribusDissertation}, Lemma 4.2.3.1, fills in some ``obvious" details.  We note that the theorem can be immediately strengthened; for example, by replacing \'etale excision with Zariski excision. 
\end{proof}

The rows of the $E_1$-page of the coniveau spectral sequence are {\it Cousin complexes} in the sense of Hartshorne \cite{HartshorneResiduesDuality66}.   Sheafifying these rows produces complexes of sheaves on $X$ of the following form:

\begin{pgfpicture}{0cm}{0cm}{17cm}{1.2cm}
\begin{pgftranslate}{\pgfpoint{1.3cm}{-.7cm}}
\pgfputat{\pgfxy(-.6,.95)}{\pgfbox[center,center]{(7.1.2)}}
\begin{pgfmagnify}{.95}{.95}
\pgfputat{\pgfxy(.7,1)}{\pgfbox[center,center]{$0$}}
\pgfputat{\pgfxy(2.7,.9)}{\pgfbox[center,center]{$\displaystyle\coprod_{x\in X^{0}} \underline{h^q(X\tn{ \footnotesize{on} } x)}$}}
\pgfputat{\pgfxy(6.3,.9)}{\pgfbox[center,center]{$\displaystyle\coprod_{x\in X^{1}} \underline{h^{q+1}(X\tn{ \footnotesize{on} } x)}$}}
\pgfputat{\pgfxy(12.4,.9)}{\pgfbox[center,center]{$\displaystyle\coprod_{x\in X^{d}} \underline{h^{q+d}(X\tn{ \footnotesize{on} } x)}$}}
\pgfputat{\pgfxy(14.9,1)}{\pgfbox[center,center]{$0$.}}
\pgfputat{\pgfxy(4.6,1.35)}{\pgfbox[center,center]{\small{$d_1^{0,q}$}}}
\pgfputat{\pgfxy(8.4,1.35)}{\pgfbox[center,center]{\small{$d_1^{1,q}$}}}
\pgfputat{\pgfxy(10.5,1.35)}{\pgfbox[center,center]{\small{$d_1^{d-1,q}$}}}
\pgfsetendarrow{\pgfarrowlargepointed{3pt}}
\pgfxyline(.9,1)(1.4,1)
\pgfxyline(4.2,1)(4.8,1)
\pgfxyline(8,1)(8.6,1)
\pgfnodecircle{Node0}[fill]{\pgfxy(8.8,1)}{0.02cm}
\pgfnodecircle{Node0}[fill]{\pgfxy(8.9,1)}{0.02cm}
\pgfnodecircle{Node0}[fill]{\pgfxy(9,1)}{0.02cm}
\pgfnodecircle{Node0}[fill]{\pgfxy(9.6,1)}{0.02cm}
\pgfnodecircle{Node0}[fill]{\pgfxy(9.7,1)}{0.02cm}
\pgfnodecircle{Node0}[fill]{\pgfxy(9.8,1)}{0.02cm}
\pgfxyline(10,1)(10.8,1)
\pgfxyline(14.1,1)(14.6,1)
\end{pgfmagnify}
\end{pgftranslate}
\end{pgfpicture}
\label{equsheafifiedcousin}

\addtocounter{equation}{1}

These complexes are of central importance in our approach to the infinitesimal theory of Chow groups.  In particular, they are flasque resolutions of the sheaves on $X$ induced by $h$ in the sense described below, and are therefore suitable for computing the corresponding sheaf cohomology objects.  Of course, the objects of principal interest in this context are the groups $H^p\big(X,\ms{K}_p(X)\big)\cong CH^p(X)$, along with the corresponding ``arc groups" and ``tangent groups."  The fact that these complexes are flasque resolutions is established by the Bloch-Ogus theorem:

\begin{theorem}\label{corblochogus}(Bloch-Ogus Theorem). Let $k$ be an infinite field, and let $\textsf{S}_k$ be a category of schemes over $k$ satisfying the conditions given at the beginning of section \hyperref[subsectioncohomsupports]{\ref{subsectioncohomsupports}}.  Let $h$ be a cohomology theory with supports on the category of pairs $\textsf{P}_k$ over $\textsf{S}_k$, satisfying \'etale excision and the projective bundle condition.  Then, for any nonsingular scheme $X$ belonging to $\textsf{S}_k$, the sheafified Cousin complexes appearing in equation \hyperref[equsheafifiedcousin]{7.1.2} above are flasque resolutions of the sheaves $\ms{H}_X^n$ associated to the presheaves $U\mapsto h_U^n$, and the $E_2$-terms of the coniveau spectral sequence for $H$ on $X$ are
\begin{equation}\label{equblochogus}E_2^{p,q}=H_{\tn{\footnotesize{Zar}}}^p(X,\ms{H}_X^q).\end{equation}
\end{theorem}
\begin{proof}  The original proof, in the case where $h$ is \'etale cohomology, appears in Bloch and Ogus \cite{BlochOgus75}.   Gabber \cite{GabberPfBlochOgus} later strengthened and extended the theorem.  Colliot-Th\'el\`ene, Hoobler and Kahn \cite{CHKBloch-Ogus-Gabber97} recast the theorem in the general context of cohomology theories with supports.  In particular, for the statement about the $E_2$-terms, see \cite{CHKBloch-Ogus-Gabber97} Corollary 5.1.11.  
\end{proof}

We pause here to make a few contextual and historical remarks.  There exists an {\it a priori} different spectral sequence involving $X$ and $h$, called the {\it Brown-Gersten spectral sequence,} or {\it descent spectral sequence,} which takes the form
\begin{equation}\label{descentSS}E_{2}^{p,q} = H_{\tn{Zar}}^{p}(X,\ms{H}^q) \Rightarrow h^{p+q}(X),\end{equation}
under appropriate conditions.   The Bloch-Ogus theorem implies that the coniveau spectral sequence coincides with the descent spectral sequence from the $E_2$-page onwards, under the stated hypothesis.   The original proof by Bloch and Ogus relies on an early version of the effacement theorem, which is proven via a ``geometric presentation lemma."  Gabber's paper \cite{GabberPfBlochOgus}, appearing nearly twenty years later, supplies a different proof of the effacement theorem, in the special case of \'etale cohomology, by considering the {\it section at infinit} associated with an embedding of the affine line into the projective line, together with a computation of the cohomology of the projective line.  Colliot-Th\'el\`ene, Hoobler and Kahn \cite{CHKBloch-Ogus-Gabber97} show how to axiomatize Gabber's argument to apply to any cohomology theory with supports satisfying a few simple technical conditions.  They provide a list of such theories, but many more have come to light subsequently. 


\subsection{First Column of the Machine.}\label{subsectionfirstcolumn}  The first column of the coniveau machine appearing in diagram \hyperref[figconiveauprecise2]{4.1.1} above is merely the Bloch-Gersten-Quillen resolution appearing in proposition 5.6 of Quillen's first paper on higher algebraic $K$-theory \cite{QuillenHigherKTheoryI72}.   As we already remarked in section \hyperref[subsecGGconiveau]{\ref{subsecGGconiveau}} above, the nonsingularity of $X$ implies that the groups with supports $K_{p-d}\big(X \tn{ on } x\big)$ in the first column may be replaced by the groups $K_{p-d}\big(k(x)\big)$ by Quillen's {\it d\'evissage} theorem, where $k(x)$ is the residue field of the local ring $\ms{O}_{X, x}$.  We formalize the existence of the first column of the coniveau machine in the following lemma:

\begin{lemma}\label{lemfirstcolumn} Let $X$ be a nonsingular quasiprojective variety of dimension $n$ over a field $k$ of characteristic zero.  Then the first column of the coniveau machine appearing in theorem \hyperref[maintheorem]{\ref{maintheorem}} exists; that is, the Cousin complexes arising as the $E_1$-rows of the coniveau spectral sequence for algebraic $K$-theory on $X$ sheafify to yield flasque resolutions of the sheaves $\ms{K}_p$ on $X$. 
\end{lemma}
\begin{proof} This result is due to Bloch, Gersten and Quillen.  Details appear in section 5 of Quillen \cite{QuillenHigherKTheoryI72}.  
\end{proof}

Though we state lemma \hyperref[lemfirstcolumn]{\ref{lemfirstcolumn}} in the special context of our main theorem \hyperref[maintheorem]{\ref{maintheorem}}, the lemma may be generalized in an obvious way to apply to any effaceable cohomology theory with supports $h$ and any nonsingular scheme $X$ belonging to a distinguished category $\textsf{S}_k$ of schemes over $k$ satisfying the conditions given at the beginning of section \hyperref[subsectioncohomsupports]{\ref{subsectioncohomsupports}}.


\subsection{Second and Thirds Columns of the Machine.}\label{subsectionfirstcolumn}  The existence of the first column of the coniveau machine requires the hypothesis that $X$ is nonsingular, appearing in the statement of the Bloch-Ogus theorem \hyperref[corblochogus]{\ref{corblochogus}}.  For the second column of the machine, this hypothesis can no longer be invoked {\it a priori}, since the second column involves the singular thickened scheme $X_A$.  The obvious way around this difficulty is to use the augmented theory $h^S$ introduced in section \hyperref[subsectionaugmented]{\ref{subsectionaugmented}} above, where in this case $h$ is algebraic $K$-theory, and $S$ is the affine scheme $\tn{Spec}(A)$ corresponding to an Artinian local $k$-algebra $A$ with residue field $k$.  

 \begin{lemma}\label{lemsecondcolumn} Let $X$ be a nonsingular quasiprojective variety of dimension $n$ over a field $k$ of characteristic zero, and let $A$ be an Artinian local $k$-algebra with residue field $k$.  Then the second column of the coniveau machine appearing in theorem \hyperref[maintheorem]{\ref{maintheorem}} exists; that is, the Cousin complexes appearing as the $E_1$-rows of the coniveau spectral sequence for augmented $K$-theory on $X$ with respect to $\tn{Spec}(A)$ sheafify to yield flasque resolutions of the sheaves $\ms{K}_p(X_A)$ on $X$.
\end{lemma}
\begin{proof}This follows immediately from lemma \hyperref[lemnewoutofold]{\ref{lemnewoutofold}}.  
\end{proof}

Next, we want to show that the third column of the coniveau machine is a flasque resolution of the relative $K$-theory sheaf $\ms{K}_p(X_A,\mm)$ on $X$.  In general, the ``correct" approach to relative constructions is to carry them out at a ``high level;" e.g., at the level of spectra or complexes, rather than adopting a na\"{i}ve definition at the group level.   This is because na\"{i}ve definitions of this kind generally fail to exhibit the desired functorial properties.  Relative $K$-theory, for example, is generally defined via homotopy fibers at the level of spectra, and this generally does not lead to na\"{i}ve group-level relationships.  In particular, the na\"{i}ve relationship
\begin{equation}\label{equnaiverelative}
K_p(X\tn{ on } Y,\mm)\cong \tn{ker}\big(K_p(X_A\tn{ on } Y)\rightarrow K_p(X\tn{ on } Y)\big),
\end{equation}
holds only because the canonical surjection $A\rightarrow k$ sending $\mm$ to the zero ideal is split by the inclusion map $k\rightarrow A$.   The consequences of this splitting extend in an obvious way to the levels of sheaves and complexes, permitting a na\"{i}ve treatment of the third column of the coniveau machine, as well as the maps of complexes between the first and second columns and the second and third columns.  

\begin{lemma}\label{lemthirdcolumn} Let $X$ and $A$ satisfy the hypotheses of lemma \hyperref[lemsecondcolumn]{\ref{lemsecondcolumn}}.  Then the third column of the coniveau machine appearing in theorem \hyperref[maintheorem]{\ref{maintheorem}} exists; that is, the kernel of the map of complexes between the second and first columns induced by the canonical splitting $A\rightarrow k$ is a flasque resolution of the sheaf $\ms{K}_p(X_A,\mm)$ on $X$.  Moreover, the splitting induces sheaf morphisms $i$ between corresponding terms of the first and second columns, which assemble to yield an injective morphism of complexes, and surjective sheaf morphisms $j$ between corresponding terms of the second and third columns, which assemble to yield a surjective morphism of complexes, as shown in figure \hyperref[figconiveauprecise2]{4.4.1}.
\end{lemma}
\begin{proof}These statements are elementary consequences of the canonical split surjection $A\rightarrow k$.   
\end{proof}

As in the case of lemma \hyperref[lemfirstcolumn]{\ref{lemfirstcolumn}} above, lemmas \hyperref[lemsecondcolumn]{\ref{lemsecondcolumn}} and \hyperref[lemthirdcolumn]{\ref{lemthirdcolumn}}  may be generalized in an obvious way to apply to any effaceable cohomology theory with supports $h$ and any nonsingular scheme $X$ belonging to an appropriate category $\textsf{S}_k$.


\subsection{Generalized Deformation Groups of Chow Groups}\label{subsectiongeneralizeddeftan}

In section \hyperref[subsecFormalCompletions]{\ref{subsecFormalCompletions}}, we defined formal completions of Chow groups.   Following Stienstra \cite{Stienstra83}, we adopted a viewpoint in which we fix the scheme $X$, and view the formal completion as a functor $\widehat{CH}^p (X)$ on the category $\textsf{Art}_k$ of Artinian local $k$-algebras with residue field $k$.  For a specific algebra $A$ in  $\textsf{Art}_k$, the sheaf cohomology group $H^p\big(X,\ms{K}_p(X_A,\mm)\big)\cong\widehat{CH}^p (X)(A)$  may be viewed as a ``generalized tangent group."   In this context, we raised the question of what role is played by the corresponding cohomology groups $H^p\big(X,\ms{K}_p(X_A)\big)$ involving {\it augmented} $K$-theory.   The answer is that these groups serve as suitable ``rigorous versions" of Green and Griffiths' ``groups of arcs through the identity in $CH^p(X)$."  We call these groups the {\it generalized deformation groups} of $CH^p(X)$ with respect to $\tn{Spec}(A)$.  The following definition is slightly more general:

\begin{definition}\label{defigendefgroupChow} Let $X$ be a smooth algebraic variety over a field $k$, and let $S$ be a separated $k$-scheme, not necessarily smooth.  The {\bf generalized deformation group} $D_S CH^p(X)$ of $CH^p(X)$ with respect to $S$ is the $p$th Zariski sheaf cohomology group of the augmented $K$-theory sheaf $\ms{K}_p(X\times_kS)$ on $X$:
\begin{equation}\label{equgendefchow}D_S CH^p(X):=H^p\big(X,\ms{K}_p(X\times_kS)\big).\end{equation}
\end{definition}

The word ``generalized" in definition \hyperref[defigendefgroupChow]{\ref{defigendefgroupChow}} refers to the fact that the definition allows for deformations with respect to any separated scheme $S$.  In particular, these deformations need not be infinitesimal.  However, this setup remains a very special case in an important respect, since $X\times_kS$ is a product.  One may consider deformations along subvarieties more generally, and this requires different machinery.   For example, the recent paper of Patel and Ravindra \cite{PatelRavindra14} examines the case of deformations along nonsingular ample hyperplane section $Y$ in an ``ambient variety $X$," and compares information about $CH^p(X)$ and $CH^p(Y)$ in the context of the {\it weak Lefschetz conjecture for Chow groups.}


\subsection{Fourth Column of the Machine}\label{subsectionfourthcolumn}

Since negative cyclic homology, like algebraic $K$-theory, is an effaceable cohomology theory with supports on an appropriate category $\textsf{S}_k$ of $k$-schemes containing the variety $X$, the obvious analogues of lemmas \hyperref[lemfirstcolumn]{\ref{lemfirstcolumn}}, \hyperref[lemsecondcolumn]{\ref{lemsecondcolumn}}, and \hyperref[lemthirdcolumn]{\ref{lemthirdcolumn}} hold for negative cyclic homology.   Hence, we may choose to recognize three additional flasque resolutions as part of the coniveau machine: the resolutions of the sheaves of absolute negative cyclic homology $\ms{HN}_p(X)$, augmented negative cyclic homology $\ms{HN}_p(X_A)$, and relative negative cyclic homology $\ms{HN}_p(X_A,\mm)$.   Dribus \cite{DribusDissertation} approaches the construction of the coniveau machine in this manner.  However, only the last of these resolutions is directly relevant to the infinitesimal theory of Chow groups.  This is because the algebraic Chern character, discussed further below, induces isomorphisms of complexes between the resolutions for $K$-theory and negative cyclic homology only in the relative case.  The existence of these isomorphisms also depends on the fact that $A$ is generated over $k$ by nilpotent elements.   Here, we streamline the construction by suppressing the resolutions of $\ms{HN}_p(X)$ and $\ms{HN}_p(X_A)$ in the statement of our main theorem, including only the resolution of $\ms{HN}_p(X_A,\mm)$ as the fourth column of the coniveau machine.   

\begin{lemma}\label{lemfourthcolumn} Let $X$ and $A$ satisfy the hypotheses of lemma \hyperref[lemsecondcolumn]{\ref{lemsecondcolumn}}.  Then the fourth column of the coniveau machine appearing in theorem \hyperref[maintheorem]{\ref{maintheorem}} exists; that is, the kernel of the map of complexes between the flasque resolutions of $\ms{HN}_p(X_A)$ and $\ms{HN}_p(X)$ arising from the coniveau filtration is a flasque resolution of the sheaf $\ms{HN}_p(X_A,\mm)$ on $X$.
\end{lemma}
\begin{proof}This is an elementary consequence of the canonical splitting $A\rightarrow k$ and the effaceability of negative cyclic homology. 
\end{proof}


\subsection{Algebraic Chern Character}
\label{cherncharacter}

The only part of the coniveau machine remaining to construct is the isomorphism of complexes between the third and fourth columns.   As stated in our main theorem \hyperref[maintheorem]{\ref{maintheorem}}, this isomorphism is induced by the relative version of the algebraic Chern character.  Loday \cite{LodayCyclicHomology98} describes the progressive generalization of the Chern character from its original incarnation in the theory of characteristic classes in differential and algebraic geometry to its modern algebraic version.  The Chern character first appeared in the theory of complex vector bundles on manifolds in the 1940's and 1950's.  It was later abstracted and generalized, first to a map from the Grothendieck group of a ring to its cyclic homology, then to a map from higher $K$-theory to negative cyclic homology, and finally to a relative version
\begin{equation}\label{equrelChern}\tn{ch}_p:K_{p}(R,I)\rightarrow  HN_p(R,I),\end{equation}
where $R$ is a suitable ring, and $I$ is an ideal in $R$.  
A good recent treatment appears in Corti\~nas and Weibel \cite{WeibelRelativeChernNilp}.  

The relative Chern character may be understood in terms of the {\it relative Volodin construction.}   Following most of the literature, we describe this construction at the level of rings, with the obvious extensions to the levels of sheaves and complexes appearing below.  Briefly, following Loday \cite{LodayCyclicHomology98}, there exists a map of complexes
\begin{equation}\label{equloday11.4.6}C_\bullet\big(X(R,I)\big)\rightarrow\tn{ker}\big(\tn{ToT }CN_R\rightarrow \tn{ToT }CN_{R,I}\big),\end{equation}
where $X(R,I)$ is an appropriate relative Volodin-type space, $C_\bullet$ is the Eilenberg-MacLane complex, $CN_R$ is the negative cyclic bicomplex of $R$, $CN_{R,I}$ is the corresponding relative negative cyclic bicomplex, and $\tn{ToT}$ is the ``total complex" whose degree-$n$ term is $\displaystyle\prod_{p+q=n} CN_{p,q}$.   Taking homology yields homomorphisms
\begin{equation}\label{equloday11.4.7}H_p\big(X(R,I)\big)\rightarrow HN_p(R,I).\end{equation}
Composing on the left with the appropriate Hurewicz maps 
\begin{equation}\label{equlodayhurewicz}K_p(R,I)=\pi_n\big(X(R,I)^+\big)\rightarrow H_n\big(X(R,I)\big),\end{equation}
yields the desired relative Chern characters.   Here, the space $X(R,I)^+$ is given by performing the plus-construction on $X(R,I)$ with respect to an appropriate maximal perfect subgroup; see \cite{LodayCyclicHomology98}, section 11.3, page 363, for details. 

\begin{lemma}\label{lemrelchernisomfunctors} The relative algebraic Chern character maps $\tn{ch}_p:K_p(R,I)\rightarrow HN_p(R,I)$ extend to isomorphisms of functors from relative algebraic $K$-theory to relative negative cyclic homology, viewed as functors from the category of split nilpotent pairs $\tn{\textsf{Nil}}$ to the category of abelian groups.  
\end{lemma}
\begin{proof} See Corti\~nas and Weibel \cite{WeibelRelativeChernNilp}, section 6.
\end{proof}

As mentioned in section \hyperref[effaceabilityaugmented]{\ref{effaceabilityaugmented}}, one may choose from the beginning to work with supports and relative constructions simultaneously at a local level. This leads to the consideration of triples consisting of a ring and two ideals, one locally representing a closed subset, and the other encoding infinitesimal structure.   One may define ``relative Chern character maps with supports" as functors of triples in this context. 

For each $p$, the isomorphism $\tn{ch}_p$ induces an isomorphism between the third and fourth columns of the coniveau machine:

\begin{lemma}\label{lemthirdchernfourth} The relative algebraic Chern character map $\tn{ch}_p$ induces an isomorphism of complexes between the third and fourth columns of the coniveau machine.  
\end{lemma}
\begin{proof} This follows automatically in the split nilpotent case from lemma \hyperref[lemrelchernisomfunctors]{\ref{lemrelchernisomfunctors}} and the functorial properties of the Chern character and the coniveau spectral sequence. 
\end{proof}


\subsection{End of the Proof}
\label{summaryofproof}

The proof of our main theorem \hyperref[maintheorem]{\ref{maintheorem}} is established by the existence of the four columns of the coniveau machine, proven in lemmas \hyperref[lemfirstcolumn]{\ref{lemfirstcolumn}}, \hyperref[lemsecondcolumn]{\ref{lemsecondcolumn}}, \hyperref[lemthirdcolumn]{\ref{lemthirdcolumn}}, and \hyperref[lemfourthcolumn]{\ref{lemfourthcolumn}}, together with the existence of the isomorphism of complexes between the third and fourth columns, proven in lemma \hyperref[lemthirdchernfourth]{\ref{lemthirdchernfourth}}.


\section{Lambda operations.}
\label{lambdaAdams}

\subsection{Decomposition of the Coniveau Machine via Cohomology Operations}  At the end of section \hyperref[subsectionsubtleties]{\ref{subsectionsubtleties}}, we remarked that our choice to define the tangent group at the identity $TCH^p(X)$ and formal completion at the identity $\widehat{CH}^p (X)$ of the Chow group $CH^p(X)$ in terms of Bass-Thomason $K$-theory, rather than Milnor $K$-theory, leads to a richer theory than the approach of Green and Griffiths, since Milnor $K$-theory corresponds only to the top Adams eigenspace of Bass-Thomason $K$-theory.  Here, we briefly discuss how this viewpoint enables a refinement of our main theorem \hyperref[maintheorem]{\ref{maintheorem}}.   

Lambda operations and Adams operations are cohomology operations that may be used, in a wide variety of contexts, to decompose cohomology objects.  Lambda operations generalize the exterior power operation, while Adams operations are polynomials in the lambda operations.  Given a natural transformation between cohomology theories, it is natural to ask whether or not this transformation respects such  operations.  In the present context, the cohomology theories of interest are $K$-theory and negative cyclic homology, and the natural transformation of interest is the Chern character.  

 The following general definition of lambda operations is adapted from Weibel \cite{WeibelKTheory11}:

\begin{definition}\label{defilambdaring} A commutative ring $K$ is called a {\bf lambda ring} if $K$ is equipped with a family of set operations 
$\lambda^i:K\rightarrow K$, for all nonnegative integers $i$, called {\bf lambda operations}, such that for every $x,y\in K$,
\begin{enumerate}
\item $\displaystyle\lambda^0(x)=1\hspace*{.3cm}\tn{ and } \hspace*{.3cm} \lambda^1(x)=x.$
\item $\displaystyle\lambda^i(x+y)=\sum_{j=0}^i\lambda^j(x)\lambda^{i-j}(y).$
\end{enumerate}
Define $\displaystyle \lambda_t(x)$ to be the formal power series $\displaystyle \sum_{i=0}^{\infty} \lambda^i(x)t^i$. 
\end{definition}

Adams operations may be defined in terms of lambda operations, under appropriate conditions.  The following definition is again adapted from Weibel \cite{WeibelKTheory11}:

\begin{definition}\label{defiAdamsop} Let $K$ be a lambda ring, and suppose that $K$ is also an augmented $k$-algebra for some commutative ring $k$, with augmentation map $\ee:K\rightarrow k$.  Define {\bf Adams operations} $\psi^i:K\rightarrow K$ by taking $\psi^i(x)$ to be the coefficient of $t^i$ in the formal power series
\[\ee(x)-t\frac{d}{dt}\log \lambda_{-t}(x)=\ee(x)+xt+\big(x^2-2\lambda^2(x)\big)t^2+...\]
\end{definition}

Note that for the zeroth Adams operation, the image $\psi^0(x)=\ee(x)$ is viewed as an element of $K$ via the map $k\rightarrow K$ defining the $k$-algebra structure of $K$.  

The Grothendieck group $K_0(X)$ of $X$ has a well-known $\lambda$-ring structure defined in terms of exterior powers; namely, $\lambda^{k}(\ms{E}) = \bigwedge^{k}\ms{E}$ for a vector bundle $\ms{E}$ over $X$.   See, for example, Weibel \cite{WeibelKTheory11}, examples 4.1.2 and 4.1.5, for details.    These operations extend to higher $K$-theory, but we will not discuss the details here.    Instead we refer the reader to the standard literature on the subject; in particular, Quillen, Hiller, Soul\'e \cite{S}, Gillet-Soul\'e \cite{GS87}, \cite{GS99}, and Grayson\cite{Gr-1}, \cite{Gr-2}.  For corresponding results involving $K$-groups with supports, see the paper of Marc Levine \cite{LevineLambdaOperationsMotivic97}.  Due to the appearance of nontrivial negative $K$-groups in our study, we also need to extend the above Adams operations  $\psi^{k}$ to the negative range.  This can be done by descending induction, as shown by Weibel in \cite{Weibel91}.  Loday \cite{LodayCyclicHomology98} is an excellent source for the corresponding operations on negative cyclic homology.  

Corti\~{n}as, Haesemeyer, and Weibel \cite{WeibelInfCohomChernNegCyclic08} recently proved that the relative Chern character, and even the absolute Chern character, preserves lamba and Adams operations.  The following theorem is a re-wording of \cite{WeibelInfCohomChernNegCyclic08}, corollary 7.2:

\begin{theorem}\label{theoremCHWlambda} Let $Y$ be a scheme of finite type over a field $k$ of characteristic 0.  Then for every choice of $p$ and $i$, the Chern character map $\tn{ch}:K_p(X)\rightarrow  HN_p(X)$ sends the $i$th Adams eigenspace $K_p^{(i)}(X)$ of algebraic $K$-theory to the corresponding $i$th Adams eigenspace $HN_p^{(i)}(X)$ of negative cyclic homology.  In particular,
\[\psi^i\circ\tn{ch}=i\tn{ch}\circ\psi^i.\]
\end{theorem}

The weaker relative version of theorem \hyperref[theoremCHWlambda]{\ref{theoremCHWlambda}} had been ``known" for some time, but its original proof, by Cathelineau \cite{CathelineauLambdaStructures91}, suffered from an error originating in an unpublished preprint of Ogle and Weibel.  This error is corrected, and the implications are carefully explained, in \cite{WeibelInfCohomChernNegCyclic08}, particularly the appendices.   These results enable the decomposition of the coniveau machine into separate pieces, one for each piece of the corresponding decompositions of $K$-theory and negative cyclic homology.

\subsection{Adams operations in negative weight}

Since the appearance of the non-zero negative K-groups in our study, we need to extend the above Adams operations  $\psi^{k}$ to negative range.
This can be done by descending induction, which was already pointed out by Weibel in \cite{Weibel91}. 
For every integer $n \in \mathbb{Z}$,  we have the following Bass fundamental exact sequence.

\[
 ... \to K_{n}(X[t,t^{-1}] \ on \ Y[t,t^{-1}]) \to  K_{n-1}(X \ on \ Y)  \to 0.
\]

 In particular, for any $x \in K_{-1}(X \ on \ Y)$, we have $x\cdot t \in K_{0}(X[t,t^{-1}] \ on \ Y[t,t^{-1}])$, where 
$t \in K_{1}(k[t,t^{-1}])$. We have
\[
 \psi^{k}(x\cdot t ) = \psi^{k}(x)\psi^{k}(t)= \psi^{k}(x) k\cdot t.
\]

Tensoring with $\mathbb{Q}$, we have obtained Adams operations $\psi^{k}$ on $K_{-1}(X \ on \ Y)$:
\[
 \psi^{k}(x)= \dfrac{\psi^{k}(x\cdot t )}{k\cdot t}.
\]

Continuing this procedure, we obtain Adams operations on all the negative K-groups.

\subsection{Adams operations with supports.} 

 Now, let $X$ be a scheme essentially finite type over a field $k$, where $\textrm{char}(k)=0$.
For every  nilpotent sheaf of ideal $I$, we define $K(\ms{O}_X,I)$ and $HN(\ms{O}_X,I)$ as the following presheaves respectively:
\[
 U \rightarrow K(\ms{O}_X(U),I(U))
\]
and
\[
 U \rightarrow HN(\ms{O}_X(U),I(U)).
\]

We write $\mathcal{K}(\ms{O}_X,I)$ and $\mathcal{HN}(\ms{O}_X,I)$ for the presheaves of spectra whose initial spaces are 
$K(\ms{O}_X,I)$ and $HN(\ms{O}_X,I)$ respectively. 
Moreover, one define $\mathcal{K}^{(i)}(\ms{O}_X,I)$ as the homotopy fiber of the map $\psi^{k}-k^{i}$ on  
$\mathcal{K}(\ms{O}_X,I)$.
We define
 $\mathcal{HN}^{(i)}(\ms{O}_X,I)$ similarly. 

\begin{theorem}
Corti$\tilde{\textrm{n}}$as-Haesemeyer-Weibel \cite{WeibelInfCohomChernNegCyclic08}.
 The relative Chern character induces a homotopy equivalence of spectra:
\[
 ch: \mathcal{K}(\ms{O}_X,I) \simeq \mathcal{HN}(\ms{O}_X,I)
\]
and 
\[
 ch: \mathcal{K}^{(i)}(\ms{O}_X,I) \simeq \mathcal{HN}^{(i)}(\ms{O}_X,I).
\]
\end{theorem}

Let $\mathbb{H}(X,\bullet)$ denote Thomason's hypercohomology of spectra, \cite{Thomason85}. 
Since both $\mathcal{K}$ and $\mathcal{HN}$ satisfy Zariski excision, we have the following identifications:
\[
 K(X \ on \ Y)=\mathbb{H}_{Y}(X, \mathcal{K}(O_{X})), 
\]

\[
 K(X[\ee] \ on \ Y[\ee])=\mathbb{H}_{Y}(X, \mathcal{K}(O_{X}[\ee])), 
\]
and similar identifications for $\mathcal{HN}$.

Now, let $X$ be a scheme essenitally finite type over a field $k$, where $\textrm{char}(k)=0$. 
Let $Y$ be a closed subset in a scheme $X$ and $U = X - Y$. Let $(A, \mathfrak{m})$ be an Artinian local
$k$-algebra with residue field $k$.
We have the following 
nine-diagram (each column and row is a homotopy fibration):

\[
  \begin{CD}
     \mathbb{H}_{Y}(X, \mathcal{K}(\ms{O}_{X_A}, \mathfrak{m})) @>>> 
     \mathbb{H}(X, \mathcal{K}(\ms{O}_{X}, \mm)) @>>>  \mathbb{H}(U, \mathcal{K}(\ms{O}_{X_A}, \mm)) \\
     @VVV  @VVV   @VVV  \\
     \mathbb{H}_{Y}(X, \mathcal{K}(\ms{O}_{X_A})) @>>> \mathbb{H}(X, \mathcal{K}(\ms{O}_{X_A})) @>>>  
     \mathbb{H}(U, \mathcal{K}(\ms{O}_{U_A})) \\
     @VVV  @VVV   @VVV  \\
     \mathbb{H}_{Y}(X, \mathcal{K}(\ms{O}_{X})) @>>> \mathbb{H}(X, \mathcal{K}(\ms{O}_{X})) @>>> 
      \mathbb{H}(U, \mathcal{K}(\ms{O}_{U})) \\
  \end{CD}
\]
and

\[
  \begin{CD}
     \mathbb{H}_{Y}(X, \mathcal{HN}(\ms{O}_{X_A}, \mm)) @>>> \mathbb{H}(X, \mathcal{HN}(\ms{O}_{X_A}, \mm)) @>>> 
      \mathbb{H}(U, \mathcal{HN}(\ms{O}_{X_A}, \mm)) \\
     @VVV  @VVV   @VVV  \\
     \mathbb{H}_{Y}(X, \mathcal{HN}(O_{X_A})) @>>> \mathbb{H}(X, \mathcal{HN}(O_{X_A})) @>>> 
      \mathbb{H}(U, \mathcal{HN}(O_{U_A})) \\
     @VVV  @VVV   @VVV  \\
     \mathbb{H}_{Y}(X, \mathcal{HN}(\ms{O}_{X})) @>>> \mathbb{H}(X, \mathcal{HN}(\ms{O}_{X})) @>>> 
      \mathbb{H}(U, \mathcal{HN}(\ms{O}_{U})) \\
  \end{CD}
\]

Comparison of these diagrams gives:
\begin{theorem} 
$\mathbb{H}_{Y}(X, \mathcal{K}(\ms{O}_{X_A}, \mm))$ is the homotpy fibre of 
\[
 \mathbb{H}_{Y}(X, \mathcal{K}(\ms{O}_{X_A})) \rightarrow \mathbb{H}_{Y}(X, \mathcal{K}(\ms{O}_{X})),
\]
and $\mathbb{H}_{Y}(X, \mathcal{HN}(\ms{O}_{X_A}, \mm))$ is the homotopy fiber of
\[
 \mathbb{H}_{Y}(X, \mathcal{HN}(\ms{O}_{X_A})) \rightarrow \mathbb{H}_{Y}(X, \mathcal{HN}(\ms{O}_{X})).
\]
\end{theorem}

\begin{corollary}
Let $K_{n}(X_A \ on \ Y_A, \mm)$ denote the kernel of 
\[
  K_{n}(X_A \ on \ Y_A) \rightarrow K_{n}(X \ on \ Y)
\]
and $HN_{n}(X_A \ on \ Y_A, \mm)$ denote the kernel of 
\[
  HN_{n}(X_A \ on \ Y_A) \rightarrow HN_{n}(X \ on \ Y),
\]
then for all n, we have an isomorphism:
\[
 K_{n}(X_A \ on \ Y_A, \mm) \cong HN_{n}(X_A \ on \ Y_A, \mm).
\]
\end{corollary}
We can generalize these results to the Adams eigenspaces.
According to [6], there exists  the following two split fibrations:
\[
 \mathcal{K}^{(i)}(\ms{O}_{X_A}, \mm) \rightarrow \mathcal{K}(\ms{O}_{X_A}, \ee) \rightarrow
  \prod_{j\neq i}\mathcal{K}^{(j)}(\ms{O}_{X_A}, \mm),
\]
and
\[
 \mathcal{HN}^{(i)}(\ms{O}_{X_A}, \mm) \rightarrow \mathcal{HN}(\ms{O}_{X_A}, \mm) \rightarrow 
 \prod_{j\neq i}\mathcal{HN}^{(j)}(\ms{O}_{X_A}, \mm).
\]

Since taking $\mathbb{H}_{Y}(X,-)$ perserves homotopy fibrations, there exists the following two split fibrations:
 \[
  \mathbb{H}_{Y}(X, \mathcal{K}^{(i)}(\ms{O}_{X_A}, \mm)) \to \mathbb{H}_{Y}(X, \mathcal{K}(\ms{O}_{X_A}, \mm))  
  \xrightarrow{\psi^{k}-k^{i}}    \mathbb{H}_{Y}(X ,\prod_{j\neq i}\mathcal{K}^{(j)}(\ms{O}_{X_A}, \mm)),
 \]
\[
 \mathbb{H}_{Y}(X, \mathcal{HN}^{(i)}(\ms{O}_{X_A}, \mm)) \to \mathbb{H}_{Y}(X, \mathcal{HN}(\ms{O}_{X_A}, \mm))  \xrightarrow{\psi^{k}-k^{i+1}}    \mathbb{H}_{Y}(X,\prod_{j\neq i}\mathcal{HN}^{(j)}(\ms{O}_{X_A}, \mm )).
\]

Passing to the group level, we obtain the following results:
\begin{theorem}
\[
 \mathbb{H}^{-n}_{Y}(X, \mathcal{K}^{(i)}(\ms{O}_{X_A}, \mm)) = 
 \{x \in \mathbb{H}^{-n}_{Y}(X, \mathcal{K}(\ms{O}_{X_A}, \mm))| \psi^{k}(x)-k^{i}(x)=0 \}.
\]
\[
 \mathbb{H}^{-n}_{Y}(X, \mathcal{HN}^{(i)}(\ms{O}_{X_A}, \mm)) =
  \{x \in \mathbb{H}^{-n}_{Y}(X, \mathcal{HN}(\ms{O}_{X_A}, \mm))| \psi^{k}(x)-k^{i+1}(x)=0 \}.
\]
\end{theorem}

We know that
\[
 \mathbb{H}^{-n}_{Y}(X, \mathcal{K}(\ms{O}_{X_A}, \mm)) = K_{n}(X_A \ on \ Y_A, \mm),
\]
and
\[
 \mathbb{H}^{-n}_{Y}(X, \mathcal{HN}(\ms{O}_{X_A}, \mm)) = HN_{n}(X_A \ on \ Y_A, \mm).
\]
Therefore, the homotopy equivalences
\[ 
  \mathcal{K}(\ms{O}_{X_A}, \mm) \simeq \mathcal{HN}(\ms{O}_{X_A}, \mm)
\]
and
\[
 \mathcal{K}^{(i)}(\ms{O}_{X_A}, \mm) \simeq\mathcal{HN}^{(i)}(\ms{O}_{X_A}, \mm),
\]
give us the following:

\begin{theorem} For all integers i and n, we have an isomorphism induced by the Chern character:
\[
 K_{n}^{(i)}(X_A \ on \ Y_A,\mm) = HN_{n}^{(i)}(X_A\ on \ Y_A,\mm).
\]
\end{theorem}


\section{Special case: the dual numbers}
\label{specialdual}

\begin{theorem}
Let $X$ be a smooth scheme over a field $k$, $chark=0$. we have the following

\begin{equation}
\begin{cases}
 \begin{CD}
 HN_{n}^{(i)}(O_{X}[\ee],\ee)= \Omega^{{2i-n-1}}_{O_{X}/ \mathbb{Q}}, for \  [\frac{n}{2}] < i \leq n.\\
 HN_{n}^{(i)}(O_{X}[\ee],\ee)= 0, else.
 \end{CD}
\end{cases}
\end{equation} 

It follows that 
\[
 HN_{n}(O_{X}[\ee],\ee) = \Omega^{{n-1}}_{O_{X}/ \mathbb{Q}}\oplus \Omega^{{n-3}}_{O_{X}/ \mathbb{Q}} \oplus \dots
\]
the last term is $\Omega^{{1}}_{O_{X}/ \mathbb{Q}}$ or $O_{X}$, depending on $n$ odd or even.
\end{theorem}
This is the sheaf version of the following theorem:
\begin{theorem}
Let  $R$ be a regular noetherian ring and also a commutative $\mathbb{Q}$-algebra and $\varepsilon ^2 = 0$. We consider $R[\varepsilon]= R \oplus \varepsilon R $ as a graded $\mathbb{Q}$-algebra.
\[
 HC_{n}(R[\varepsilon],\varepsilon) = \Omega^{{n}}_{R/ \mathbb{Q}}\oplus \Omega^{{n-2}}_{R/ \mathbb{Q}} \oplus \dots
\]
the last term is $\Omega^{{1}}_{R/ \mathbb{Q}}$ or $R$, depending on $n$ odd or even.
\end{theorem}

Before proving this, let us note the corollaries. 

 


Note that for any commutative $k$-algebra  $A$, 
where $k$ is a field of characteristic $0$, and $I$ be an ideal of $A$,
\[
 HN_{n}(A,I)=HC_{n-1}(A,I).
\]
\[
 HN_{n}^{(i)}(A,I)=HC_{n-1}^{(i-1)}(A,I).
\]

We now outline the proof of this theorem. Let $A$ be any commutative $\mathbb{Q}$-algebra and $I$ be an ideal of $A$.
We can associate a  Hochschild complexes $C^{h}_{\ast}(A)$ to $A$ as in \cite{Lodayoper89}, \cite{LodayCyclicHomology98}. 
The action of the symmetric groups on $C^{h}_{\ast}(A)$
gives the lambda operation
\[
 HH_{n}(A)=HH_{n}^{(1)}(A) \oplus \dots \oplus HH_{n}^{(n)}(A),
\]
and similarly
\[
 HC_{n}(A)=HC_{n}^{(1)}(A) \oplus \dots \oplus HC_{n}^{(n)}(A),
\]
\[
 HN_{n}(A)=HN_{n}^{(1)}(A) \oplus \dots \oplus HN_{n}^{(n)}(A).
\]

 There is also a Hochschild complexes  
$C^{h}_{\ast}(A/I)$ associated to  $A/I$. We use $C^{h}_{\ast}(A,I)$ to denote the kernel of the natural map
\[
 C^{h}_{\ast}(A) \rightarrow C^{h}_{\ast}(A/I).
\]
Then the relative Hochschild module $HH_{\ast}(A,I)$ is the homology of the complex $C^{h}_{\ast}(A,I)$. Moreover, the action of the symmetric groups on $C^{h}_{\ast}(A,I)$
gives the lambda operation
\[
 HH_{n}(A,I)=HH_{n}^{(1)}(A,I) \oplus \dots \oplus HH_{n}^{(n)}(A,I)
\]
and similarly
\[
 HC_{n}(A,I)=HC_{n}^{(1)}(A,I) \oplus \dots \oplus HC_{n}^{(n)}(A,I),
\]
\[
 HN_{n}(A,I)=HN_{n}^{(1)}(A,I) \oplus \dots \oplus HN_{n}^{(n)}(A,I).
\]

If $R$ is a regular noetherian ring and also a commutative $\mathbb{Q}$-algebra, and $\varepsilon ^2 = 0$, then we have 
the following SBI sequence is obtained from the corresponding eigenspace of the relative Hochschild complex:
\[
 \rightarrow HC^{(i)}_{n+1}(R[\varepsilon],\varepsilon) \xrightarrow{S} HC^{(i-1)}_{n-1}(R[\varepsilon],\varepsilon) \xrightarrow{B} HH^{(i)}_{n}(R[\varepsilon],\varepsilon) \xrightarrow{I} HC^{(i)}_{n}(R[\varepsilon],\varepsilon) \rightarrow
\]

According to a result of Geller-Weibel, \cite{GW94}, the above S map is $0$ on $HC(R[\varepsilon],\varepsilon)$. This enable us to break the SBI sequence up into a
short exact sequence:
\[
 0 \rightarrow HC^{(i-1)}_{n-1}(R[\varepsilon],\varepsilon) \xrightarrow{B} HH^{(i)}_{n}(R[\varepsilon],\varepsilon) \xrightarrow{I} HC^{(i)}_{n}(R[\varepsilon],\varepsilon) \rightarrow 0.
\]
In the following, we will use this short exact sequence to compute $HC^{(i)}_{n}(R[\varepsilon],\varepsilon)$.

\begin{proposition}
\begin{equation}
\begin{cases}
 \begin{CD}
 HC_{n}^{(i)}(R[\varepsilon],\varepsilon)= \Omega^{{2i-n}}_{R/ \mathbb{Q}}, for \  [\frac{n}{2}] \leq i \leq n.\\
 HC_{n}^{(i)}(R[\varepsilon],\varepsilon)= 0, else.
 \end{CD}
\end{cases}
\end{equation} 

\end{proposition}

\begin{proof}
 Easy exercise which can be done  by computing $HH^{(i)}_{n}(R[\varepsilon],\varepsilon)$ and then use induction. More details can be found in \cite{Yang12}. \end{proof}

We connect these general results with those of Green and Griffiths.
\begin{theorem}
Suppose $X$ is a $d$-dimensional smooth variety over a field $k$, where $\textrm{char} k=0$ and $y \in X^{(j)}$.  For any integer $m$, we have 
\[
 HN_{m}(O_{X,y}[\ee] \ on \ y[\ee],\ee)= H_{y}^{j}(\Omega^{\bullet}_{O_{X,y}/\mathbb{Q}}),
\]
where $\Omega^{\bullet}_{O_{X,y}/\mathbb{Q}}=\Omega^{m+j-1}_{O_{X,y}/\mathbb{Q}}\oplus \Omega^{m+j-3}_{O_{X,y}/\mathbb{Q}}\oplus \dots$
\end{theorem}

\begin{proof}
$O_{X,y}$ is a regular local ring with dimension $j$, so the depth of $O_{X,y}$ is $j$. For each $n \in \mathbb{Z}$,  $\Omega^{n}_{O_{X,y}/\mathbb{Q}}$ can be written as a direct limit of rings
$O_{X,y}$. Therefore, $\Omega^{n}_{O_{X,y}/\mathbb{Q}}$ has depth $j$.

Let $HN_{m}(O_{X,y}[\ee] \ on \ y[\ee],\ee)$ be the kernel of the projection:
\[
HN_{m}(O_{X,y}[\ee] \ on \ y[\ee])  \xrightarrow{\ee =0} HN_{m}(O_{X,y} \ on \ y).
\]
Then $HN_{m}(O_{X,y}[\ee] \ on \ y[\ee],\ee)$ can be identified with $\mathbb{H}_{y}^{-m}(O_{X,y},HN(O_{X,y}[\ee],\ee))$,
where $HN(O_{X,y}[\ee],\ee)$ is the relative negative cyclic complex, that is the kernel of
\[
 HN(O_{X,y}[\ee]) \xrightarrow{\ee=0} HN(O_{X,y}).
\]

There is a spectral sequence :
\[
 H_{y}^{p}(O_{X,y}, H^{q}(HN(O_{X,y}[\ee],\ee))) \Longrightarrow \mathbb{H}_{y}^{-m}(HN(O_{X,y}[\ee],\ee)).
\]

By corollary 4.1.4, we have
 \[
 H^{q}(HN(O_{X,y}[\ee],\ee))= HN_{-q}(O_{X,y}[\ee],\ee)= \Omega^{-q-1}_{O_{X,y}/\mathbb{Q}}\oplus \Omega^{-q-3}_{O_{X,y}/\mathbb{Q}}\oplus \dots
 \]
 As each $\Omega^{n}_{O_{X,y}/\mathbb{Q}}$ has depth $j$, only $H_{y}^{j}(X,H^{q}(HN(O_{X,y}[\ee],\ee)))$ can survive because of the depth condition.
This means $q=-m-j$ and 
\[
 H^{-m-j}(HN(O_{X,y}[\ee],\ee))=HN_{m+j}(O_{X,y}[\ee],\ee)= \Omega^{m+j-1}_{O_{X,y}/\mathbb{Q}}\oplus \Omega^{m+j-3}_{O_{X,y}/\mathbb{Q}}\oplus \dots
\]
Let us denote 
\[
\Omega^{\bullet}_{O_{X,y}/\mathbb{Q}} = \Omega^{m+j-1}_{O_{X,y}/\mathbb{Q}}\oplus \Omega^{m+j-3}_{O_{X,y}/\mathbb{Q}}\oplus \dots
\]
Thus 
\[
 \mathbb{H}_{y}^{-m}(HN(O_{X,y}[\ee],\ee))=H_{y}^{j}(\Omega^{\bullet}_{O_{X,y}/\mathbb{Q}}).
\]
this means
\[
 HN_{m}(O_{X,y}[\ee] \ on \ y_{\ee},\ee)= H_{y}^{j}(\Omega^{\bullet}_{O_{X,y}/\mathbb{Q}}).
\]
\end{proof}

Repeating the above proof and noting corollary 4.1.3, we have the following finer result:
\begin{theorem}
 Suppose $X$ is a $d$-dimensional smooth variety over a field $k$, where $\textrm{char} k=0$ and $y \in X^{(j)}$. For any integer $m$, we have 
\[
 HN^{(i)}_{m}(O_{X,y}[\ee] \ on \ y[\ee],\ee)= H_{y}^{j}(\Omega^{\bullet,(i)}_{O_{X,y}/\mathbb{Q}}),
\]
where 
\begin{equation}
\begin{cases}
 \begin{CD}
 \Omega_{O_{X}/ \mathbb{Q}}^{\bullet,(i)}= \Omega^{{2i-(m+j)-1}}_{O_{X}/ \mathbb{Q}}, for \  \frac{m+j}{2}  < \ i \leq m+j.\\
  \Omega_{O_{X}/ \mathbb{Q}}^{\bullet,(i)}= 0, else.
 \end{CD}
\end{cases}
\end{equation} 
\end{theorem}

Combining with theorem 4.3.5 and 4.3.7, we have the following corollary
\begin{corollary}
Under the same assumption as above, we have
\[
 K_{m}(O_{X,y}[\ee] \ on \ y[\ee],\ee)= H_{y}^{j}(\Omega^{\bullet}_{O_{X,y}/\mathbb{Q}}),
\]
where $\Omega^{\bullet}_{O_{X,y}/\mathbb{Q}}=\Omega^{m+j-1}_{O_{X,y}/\mathbb{Q}}\oplus \Omega^{m+j-3}_{O_{X,y}/\mathbb{Q}}\oplus \dots$

\[
 K^{(i)}_{m}(O_{X,y}[\ee] \ on \ y[\ee],\ee)= H_{y}^{j}(\Omega^{\bullet,(i)}_{O_{X,y}/\mathbb{Q}}),
\]
where 
\begin{equation}
\begin{cases}
 \begin{CD}
 \Omega_{O_{X}/ \mathbb{Q}}^{\bullet,(i)}= \Omega^{{2i-(m+j)-1}}_{O_{X}/ \mathbb{Q}}, for \  \frac{m+j}{2}  < \ i \leq m+j.\\
  \Omega_{O_{X}/ \mathbb{Q}}^{\bullet,(i)}= 0, else.
 \end{CD}
\end{cases}
\end{equation} 
\end{corollary}

\begin{theorem}
\label{T:adamsdual}
 There exists the following commutative diagram, where the rows are split exact sequences and the columns are 
 flasque resolutions of their respective starting terms:
{\footnotesize
\[
  \begin{CD}
     0 @. 0 @. 0\\
     @VVV @VVV @VVV\\
     \Omega_{O_{X}/ \mathbb{Q}}^{\bullet,(i)} @<ch<< K^{(i)}_{m}(O_{X[\ee]}) @<<< K^{(i)}_{m}(O_{X}) \\
     @VVV @VVV @VVV\\
     \Omega_{k(X)/ \mathbb{Q}}^{\bullet,(i)} @<ch<<  K^{(i)}_{m}(k(X)[\ee]) @<<< K^{(i)}_{m}(k(X)) \\
     @VVV @VVV @VVV\\
     \displaystyle{\bigoplus_{d \in X^{(1)}}}\underline{H}_{d}^{1}(\Omega_{O_{X}/\mathbb{Q}}^{\bullet,(i)}) @<ch<<
     \displaystyle{\bigoplus_{d[\ee]\in X[\ee]^{(1)}}}\underline{K}^{(i)}_{m-1}(O_{X,d}[\ee] \ on \ d[\ee]) @<<< \displaystyle{\bigoplus_{d \in X^{(1)}}}\underline{K}^{(i)}_{m-1}(O_{X,d} \ on \ d)\\
     @VVV @VVV @VVV\\
     \displaystyle{\bigoplus_{y \in X^{(2)}}}\underline{H}_{y}^{2}(\Omega_{O_{X}/ \mathbb{Q}}^{\bullet,(i)}) @<ch<<
      \displaystyle{\bigoplus_{y[\ee] \in X[\ee]^{(2)}}}\underline{K}^{(i)}_{m-2}(O_{X,y}[\ee] \ on \ y[\ee]) @<<< 
      \displaystyle{\bigoplus_{y \in X^{(2)}}}\underline{K}^{(i)}_{m-2}(O_{X,y} \ on \ y) \\
     @VVV @VVV @VVV\\
      \dots @<ch<< \dots @<<< \dots \\ 
     @VVV @VVV @VVV\\
     \displaystyle{\bigoplus_{x\in X^{(n)}}}\underline{H}_{x}^{n}(\Omega_{O_{X}/ \mathbb{Q}}^{\bullet,(i)}) @<ch<< 
     \displaystyle{\bigoplus_{x[\ee]\in X[\ee]^{(n)}}}\underline{K}^{(i)}_{m-n}(O_{X,x}[\ee] \ on \ x[\ee]) @<<< 
      \displaystyle{\bigoplus_{x \in X^{(n)}}}\underline{K}^{(i)}_{m-n}(O_{X,x} \ on \ x) \\
     @VVV @VVV @VVV\\
      0 @. 0 @. 0
  \end{CD}
\]
}
where 
\begin{equation}
\begin{cases}
 \begin{CD}
 \Omega_{O_{X}/ \mathbb{Q}}^{\bullet,(i)}= \Omega^{{2i-m+1}}_{O_{X}/ \mathbb{Q}}, for \  \frac{m-1}{2}  < \ i \leq m-1.\\
  \Omega_{O_{X}/ \mathbb{Q}}^{\bullet,(i)}= 0, else.
 \end{CD}
\end{cases}
\end{equation} 
\end{theorem}


\end{document}